\documentclass[11pt,english]{smfart}

\usepackage{amsmath,xy,leftidx,amsthm,sfheaders}
\usepackage{xspace}
\usepackage[psamsfonts]{amssymb}
\usepackage[latin1]{inputenc}
\usepackage{graphicx,color}

\usepackage{hyperref}

\theoremstyle{remark}

\newtheorem*{remark}{\bf Remark}

\newtheorem*{notation}{\bf Notation}

\theoremstyle{plain}
\newtheorem{theorem}{\bf Theorem}[section]
\newtheorem{proposition}[theorem]{\bf Proposition}

\newtheorem{Theorem}{\bf Theorem}

\newtheorem{lemma}[theorem]{\bf Lemma}
\newtheorem{corollary}[theorem]{\bf Corollary}

\newtheorem*{conjecture*}{\bf  Conjecture}

\def\A{{\mathbb A}}
\def\C{{\mathbb C}}
\def\R{{\mathbb R}}
\def\Z{{\mathbb Z}}

\def\Q{{\mathbb Q}}

\def\D{\mathbb{D}}
\def\KK{\mathbb{K}}
\def\LL{\mathbb{L}}
\def\p{\mathbb{P}}
\def\N{{\mathbb N}}

\def\O{\mathcal{O}}

\def\K{\mathcal{K}}

\def\om{\omega}

\DeclareMathOperator{\ndiv}{div}

\DeclareMathOperator{\ord}{ord}

\DeclareMathOperator{\an}{an}

\DeclareMathOperator{\per}{Per}
\DeclareMathOperator{\hdiv}{\widehat{div}}

\DeclareMathOperator{\Per}{Per}

\DeclareMathOperator{\crit}{Crit}

\def\bif{\textup{bif}}
\def\loc{\textup{loc}}

\DeclareMathOperator{\poly}{Poly}

\DeclareMathOperator{\spec} {Spec}

\def\and{{\quad\text{and}\quad}}

\begin{document}

\title[Special curves for cubic polynomials]
{Classification of special curves in the space of cubic polynomials}
\author{Charles Favre}
\address{CMLS, \'Ecole polytechnique, CNRS, Universit\'e Paris-Saclay, 91128 Palaiseau Cedex, France}
\email{charles.favre@polytechnique.edu}
\author{Thomas Gauthier}
\address{LAMFA, Universit\'e de Picardie Jules Verne, 33 rue Saint Leu, 80039 Amiens Cedex}
\email{thomas.gauthier@u-picardie.fr}

\thanks{First author is supported by the ERC-starting grant project "Nonarcomp" no.307856, both authors are partially supported by ANR project ``Lambda'' ANR-13-BS01-0002}

\date{\today}

\begin{abstract}
We describe all special curves in the parameter space of complex cubic polynomials, that is all algebraic irreducible curves containing infinitely many 
post-critically finite polynomials. This solves in a strong form a conjecture by Baker and DeMarco for cubic polynomials.

 Let $\per_m(\lambda)$ be the algebraic curve 
consisting of those cubic polynomials that admit an orbit of period $m$ and multiplier $\lambda$.
We also prove that an irreducible component of $\per_m(\lambda)$
is special if and only if $\lambda =0$.
\end{abstract}

\maketitle

\tableofcontents

\section{Introduction}

The space $\poly_d$ of  complex polynomials of degree $d\ge2$ modulo affine conjugacy forms a complex analytic space that admits a ramified  parameterization
by the affine space  $\A^{d-1}_\C$. The study of the set of degree $d$ polynomials with special dynamical features forms the core of the modern theory of holomorphic dynamics. We shall be concerned here with the distribution of the set of post-critically finite (PCF) polynomials for which all critical points have a finite orbit under iteration. This set is a countable union of points defined over a number field, see e.g.~\cite[Corollary 3]{Ingram}. It was proved in~\cite{Levin1} that 
hyperbolic PCF quadratic polynomials equidistribute to the harmonic measure of the Mandelbrot set.
This convergence was generalized in~\cite{baker-hsia} in degree $2$, and later in~\cite{favregauthier} in any degree where it was proved that under a mild assumption
 any sequence of Galois-invariant finite subsets of PCF polynomials converges in the sense of measures to the so-called bifurcation measure. This fact was further explored in~\cite{gauthiervigny1}.
The support of this measure has been characterized in several ways in a series of works~\cite{favredujardin,dujardin2,dujardin-higher,gauthier-higher}, and it was shown by the second author~\cite{Article1} to have maximal
Hausdorff dimension $2(d-1)$.

\smallskip

In a beautiful recent paper~\cite{BD}, Baker and DeMarco have proposed a way to describe the distribution of PCF polynomials from the point of view of the Zariski topology. 
They defined  \emph{special} algebraic subvarieties as those subvarieties $Z\subset \poly_d$ admitting a Zariski-dense subset formed by PCF polynomials, and asked about the classification of such varieties. More precisely, they offered a quite general conjecture~\cite[Conjecture 1.4]{BD} inspired by the Andr\'e-Oort conjecture in arithmetic geometry  stating that any polynomial\footnote{the conjecture is actually stated for any rational maps, and a stronger conjecture related to Pink and Zilber's conjectures can be found in~\cite{Demarco-stable}.} lying in a special (proper) subvariety should admit a critical orbit relation (see also~\cite[Conjecture 6.56]{Silverman-modulispace}). They gave a proof of a stronger version of this conjecture in the case the subvariety was isomorphic to an affine line.

Our objective is to give the list of all special curves in the case $d=3$, thereby proving Baker-DeMarco's conjecture for cubic polynomials. 
The geometry of the space of cubic polynomials has been thoroughly explored in the seminal work~\cite{BH} of Branner and Hubbard. Instead of using their parameterization, we shall follow~\cite{favredujardin}  and consider the parameterization $(c,a) \mapsto P_{c,a}$ of the parameter space by the affine plane
with
\[P_{c,a}(z):=\frac{1}{3}z^3-\frac{c}{2}z^2+a^3, \ z\in\C, (c,a)\in \C^2~.\] 
Observe that $P_{c,a}$ then admits two critical points $c_0:= 0$ and $c_1:=c$ and that this map defines
a finite branched cover of the moduli space $\textup{Poly}_3$ of cubic polynomials with marked critical points.

~

Here is our main result.
\begin{Theorem}\label{tm:classification}
An irreducible curve $C$ in the space $\poly_3$ contains an infinite collection of post-critically finite polynomials if and only if one of the following holds.
\begin{enumerate}
\item One of the two critical points is persistently pre-periodic on $C$, i.e. there exist integers $m>0$ and $k\ge0$ such that 
 $P_{c,a}^{m+k}(c_0) = P_{c,a}^k(c_0)$ or $P_{c,a}^{m+k}(c_1) = P_{c,a}^k(c_1)$ for all $(c,a)\in C$.
\item There is a persistent collision of the two critical orbits on $C$, i.e.  there exist $(m,k)\in\mathbb{N}^2\setminus\{(1,1)\}$ such that 
$P_{c,a}^{m}(c_1) = P_{c,a}^k(c_0)$ for all $(c,a)\in C$.
\item The curve $C$ is given by the equation $\{(c,a),\, 12a^3-c^3-6c=0\}$, and coincides with the set of cubic polynomials having a non-trivial symmetry, i.e. the set of parameters $(c,a)$ for which $Q_c(z):=-z+c$ commutes with $P_{c,a}$.
\end{enumerate}
\end{Theorem}
 Recall that for any integer $m\ge1$ and any complex number $\lambda\in \C$
the set $\per_m(\lambda)$ consisting of all polynomials $P_{c,a}\in \poly_3$ that admits at least one periodic orbit of period $m$ and multiplier $\lambda$ is an algebraic curve (see \S\ref{sec:perm} for a more precise description).

The geometry of these curves has been explored by several authors, especially when $\lambda =0$. The irreducible components of $\Per_m(0)$ has been proven to be smooth by Milnor~\cite{Milnor-cubic}, and the escape components of these curves have been described in terms of Puiseux series by Bonifant, Kiwi and Milnor~\cite{BKM} (see also~\cite[\S 7]{Kiwi:cubic}). On the other hand, DeMarco and Schiff~\cite{DMS} have given an algorithm to compute their Euler characteristic.

From the point of view of pluripotential theory, the distribution of the sequence of curves $(\Per_m(\lambda))_{m\geq1}$ has been completely described by Bassanelli and Berteloot in~\cite{BB2} in the case $|\lambda|\leq1$ (see also~\cite{DistribTbif} for the case $|\lambda|>1$ and~\cite{BB3} for the case of quadratic rational maps).

Inspired by a similar result from Baker and DeMarco, see~\cite[Theorem~1.1]{BD} we also give a characterization of
those $\per_m(\lambda)$ that contain infinitely many PCF polynomials. This answers a conjecture of DeMarco in the case of cubic polynomials (see e.g.~\cite[Conjecture 6.59]{Silverman-modulispace}). More precisely, we prove
\begin{Theorem}\label{thm:perm}
For any $m\ge1$, the curve $\per_m(\lambda)$ contains infinitely many post-critically finite polynomials if and only if $\lambda =0$.
\end{Theorem}

The general strategy of the proof of these two theorems was set up by Baker and DeMarco. 
They start with an irreducible algebraic curve $C\subset \poly_3$ containing infinitely many PCF polynomials (in Theorem~\ref{thm:perm} the curve $C$ is a component
of some $\per_m(\lambda)$).  We observe however that they used at several key points their assumption
that  the curve $C$ has a single branch at infinity. To remove this restriction we had to include two new ingredients:
\begin{itemize}
\item
we construct a one parameter family of heights for which Thuillier-Yuan's equidistribution theorem~\cite{yuan,thuillier} applies;
\item 
we investigate systematically the arithmetic properties of the coefficients of the expansion of the B\"ottcher coordinates and its dependence on the parameters 
$c,a$.
\end{itemize}
We propose also a new way to build the symmetry by relying on a recent algebraization result of Xie~\cite{Xie} that gives a criterion for when a formal curve in the affine plane is a branch of an algebraic curve.

\smallskip

A characteristic feature of our proofs is to look at the dynamics induced by cubic polynomials over various fields: 
over the complex numbers and over $p$-adic fields (see e.g. \S\ref{sec:propor}), over the field of Laurent series (see the proof of Proposition~\ref{prop:branch} and \S \ref{sec:perm}), and 
over a number field (see \S \ref{sec:curves in P3}).  We use at one point the universality theorem of McMullen~\cite{McMullen3} which is a purely Archimedean statement. 
Moreover the work of Kiwi~\cite{Kiwi:cubic} on non-Archimedean cubic polynomials over a field of residual characteristic zero plays a key role in the proof of Theorem~\ref{thm:perm}.

\medskip

Let us describe in more detail how we proceed, and so pick an irreducible algebraic curve $C\subset \poly_3$ containing infinitely many PCF polynomials.
We may suppose that  neither $c_0$ nor $c_1$ are persistently pre-periodic on $C$. By a theorem of McMullen~\cite[Lemma~2.1]{McMullen4} this is equivalent to 
say that both critical points exhibit bifurcations at some (possibly different) points in $C$.  There is a more quantitative way to describe the set of bifurcations
using the Green function $g_{c,a}(z) := \lim_{n\to \infty} \frac1{3^n} \log \max \{ 1, | P^n(z)|\}$. Indeed both  functions $g_0(c,a) := g_{c,a}(c_0)$, $g_1(c,a) := g_{c,a}(c_1)$ are non-negative and pluri-subharmonic, and it is a fact~\cite[\S 5]{Demarco1} that the support of the positive measure $\Delta g_0|_C$ (resp. $\Delta g_1|_C$) is equal to the set of parameters where $c_0$ (resp. $c_1$) is unstable. 

\smallskip

The first step consists in proving that $g_0|_C$ and $g_1|_C$ are proportional, and this conclusion is obtained by applying an equidistribution result of points of small height due to Yuan~\cite{yuan} and Thuillier~\cite{thuillier}. We first observe  that $C$ is necessarily defined over a number field $\KK$ since it contains infinitely many PCF polynomials, so that  we may introduce the functions $g_{0,v}, g_{1,v}$ for all (not necessarily Archimedean) places $v$ over $\KK$. 
These functions can now be used to build a one-parameter family of heights on $C$ by setting
\[h_s(p) := \frac1{\deg(p)} \sum_{q,v} \max\{s_0 g_{0,v}(q), s_1 g_{1,v}(q)\} \]
where $s = (s_0, s_1)\in \R_+^2$, and the sum ranges over all Galois conjugates $q$ of $p$ and over all places $v$ over $\KK$.
When $s_0$ and $s_1$ are positive integers, then we prove in \S\ref{sec:curves in P3} that the height $h_s$ is induced
by a continuous semi-positive adelic metrization in the sense of Zhang on a suitable line bundle over $C$ of positive degree,
 so that Thuillier-Yuan's theorem 
applies. This gives us sufficiently many  restrictions on $g_0$ and $g_1$ which force their proportionality.
The key arguments are Proposition~\ref{prop:branch} that is close in spirit to~\cite[Proposition~2.1 (3)]{BD}, and the fact that the function
$\max\{g_{0,v}, g_{1,v}\}$ is a proper continuous function on $\poly_3$ for any place $v$.

\medskip

From the proportionality of $g_0$ and $g_1$ on a special curve, we are actually able to conclude the proof of Theorem~\ref{thm:perm}. This step is done in \S\ref{sec:perm}. We suppose by contradiction that our special curve $C$ is an irreducible component of some $\per_m(\lambda)$ with $\lambda \neq 0$. 
Then each branch at infinity of $C$ defines a cubic polynomial over the  field of Laurent series $\C((t))$. 
We show that 
except when $c_0$ or $c_1$ is persistently periodic in $C$ the multipliers of all periodic points are exploding on that branch by~\cite{Kiwi:cubic}.
We then analyze the situation of a unicritical\footnote{i.e. having a single critical point} polynomial in $C$ and computing the norm of the multiplier of its periodic points in a suitable field
of residual characteristic $3$, we are able to get the required contradiction.

\medskip

Let us come back to the proof of Theorem~\ref{tm:classification}. At this point, 
we have an irreducible algebraic curve $C$ defined over a number field $\KK$ and such that $g_{0,v} = g_{1,v}$ at any place $v$ over $\KK$.
Recall that for any polynomial $P_{c,a}$ there exists an analytic isomorphism near infinity conjugating the polynomial to the cubic monomial map. 
This isomorphism is referred to as the B\"ottcher coordinate $\varphi_{c,a}$ of $P_{c,a}$. We prove that when $c,a$ are defined over a number field
then $\varphi_{c,a}$ is a power series with coefficients in a number field whose domain of convergence is positive at any place, see Lemma~\ref{lm:bottcher1} and Proposition~\ref{prop:green is green}.

Building on an argument of Baker and DeMarco, we then show that outside a compact subset of the analytification of  $C$ (for any completion of $\KK$)
the values of the B\"ottcher coordinates at $c_0$ and $c_1$ are proportional up to a root of unity (Theorem~\ref{tm:samegreen} (2)).

\smallskip

The proof now takes a slight twist as we fix any polynomial $P:=P_{c,a}$ that is \emph{not} post-critically finite and for which $(c,a)$ belongs to $C(\LL)$
for some finite field extension $\LL$ of $\KK$. We prove that any such polynomial admits a weak form of symmetry
in the sense that  there exists an irreducible curve $Z_P\subset \p^1 \times \p^1$ that is stable by the map $(P,P)$.
To do so we apply~\cite[Theorem~1.5]{Xie} as an alternative to the arguments of Baker and DeMarco in~\cite[\S 5.6]{BD}.
In order to get a polynomial that commutes with $P$ instead of a correspondence, we proceed as Baker and DeMarco and
use Medvedev-Scanlon's result~\cite[Theorem~6.24]{medvedev-scanlon} (see~\cite[Theorem 4.9]{pakovich} for another proof of this result). 

~

At this point we have proved the following result that we feel is of independent interest. 
\begin{Theorem}\label{tm:alaBDM}
Pick any irreducible complex algebraic curve $C\subset\poly_3$. Then the following assertions are equivalent:
\begin{enumerate}
\item the curve $C$ is special,
\item for any critical point that is not persistenly pre-periodic on $C$, 
the set of PCF polynomials lying in $C$ is equal to the set where this critical point is pre-periodic;
\item the curve $C$ is defined over a number field $\KK$ and there exist integers $(s_0,s_1)\in\mathbb{N}^2\setminus\{(0,0)\}$ 
such that for any place $v\in M_\KK$, we have
\[s_0\cdot g_{0,v} = s_1\cdot g_{1,v}~,\] 
on the analytification of $C$ over the completion of $\KK$ w.r.t. the $v$-adic norm;
\item the curve $C$ is special, and for any sequence $X_k\subset C$ of Galois-invariant finite sets of PCF polynomials with $X_k\neq X_l$ for $l\neq k$, the probability measures $\mu_k$ equidistributed on $X_k$ converge towards (a multiple of) the bifurcation measure $T_\bif\wedge[C]$  as $k\to\infty$;
\item there exists a root of unity $\zeta$, and integers $q,m\ge0$ such that the polynomial $Q_{c,a}(z) := \zeta P_{c,a}^m(z) + (1-\zeta) \frac{c}{2}$ commutes with 
any iterate $P_{c,a}^k$ such that $\zeta^{3^k} = \zeta$, and $Q_{c,a}(P^q_{c,a}(c_i)) = P^q_{c,a}(c_j)$ for some $i,j\in\{0,1\}$ and all $(c,a)\in C$.
\end{enumerate}
\end{Theorem}
In (4) the current $T_\bif$ is defined as the $dd^c$ of the plurisubharmonic function $g_0+g_1$. Its support in $\C^2$ is known to be equal to the set of unstable parameters, see e.g.~\cite[\S3]{favredujardin}. Notice that for any curve $C$ there exists a critical point which is not persistently pre-periodic on $C$ since by~\cite{BH} the set $\{g_0=0\}\cap\{g_1=0\}$ is \emph{compact} in $\A^2_\C$. In particular, the assertion (2) is consistent.

\smallskip

To complete the proof of Theorem~\ref{tm:classification}, we analyze in more detail the possibilities for a cubic polynomial
to satisfy the condition (5) in the previous theorem. Namely, we prove that the set of parameters admitting a non-trivial symmetry of degree $3^m>1$ is actually \emph{finite}. Theorems~\ref{tm:classification} and~\ref{tm:alaBDM} are proved in \S \ref{sec:thmA}.

~

We have deliberately chosen  to write the entire paper for  cubic polynomials only. This simplifies the exposition, but many parts of the proof actually 
extend to a larger context.
Let us briefly discuss the possible extensions and the limitations of our approach. 

All ingredients are present to prove Baker-DeMarco's conjecture for a \emph{curve} in the space of polynomials of any degree $d\ge2$. It is however not clear to the authors how to obtain the more precise classification of special curves in the same vein as in Theorem~\ref{tm:classification}.

\medskip

We note that there are serious difficulties that lie beyond the methods presented here to handle  higher dimensional special varieties $V$ in $\poly_d$.
The main issue is the following. To apply Yuan's equidistribution theorem of points of small heights it is necessary to have 
a \underline{continuous} \underline{semi-positive} adelic metrics on an \underline{ample} line bundle on a compactification of $V$, 
and we are at the moment very far from being able to check any of the three underlined conditions.

Trying to understand special curves in the space of quadratic maps requires  much more delicate estimates than in the case of polynomials. 
A first important step has been done by DeMarco, Wang and Ye in a recent paper~\cite{DWY}.

\paragraph*{Acknowledgements}
We thank  Xavier Buff  and Laura DeMarco for discussions at a preliminary stage of this project, and the referee for his/her careful reading of this paper and his/her constructive remarks.

While finishing the writing of this paper we have learned that Dragos Ghioca and Hexi Ye have independently obtained a proof of Theorem~\ref{tm:classification}. 
Their approach differs from ours in the sense that they directly prove the continuity of the metrizations induced by the functions $g_{0,v}$ and $g_{1,v}$. 
We get around this problem by considering metrizations induced by $\max\{s_0 g_{0,v}, s_1 g_{1,v}\}$ for positive $s_0, s_1$ instead.
We warmly thank D. Ghioca and H. Ye for sharing with us their preprint.

\bigskip

\section{The B\"ottcher coordinate of a polynomial}
In this section, $K$ is any complete metrized field of characteristic zero containing a square-root $\om$ of $\frac13$.
It may or may not be endowed with a non-Archimedean norm.

If $X$ is an algebraic variety over $K$, then $X^{\an}$ denotes its analytification as a real-analytic or a complex variety if $K$ is Archimedean, 
and as a Berkovich analytic space when $K$ is non-Archimedean (see e.g.~\cite[\S 3.4-5]{Berko}).
\subsection{Basics}
As in the introduction, we denote by $\poly_3 \simeq \A^2$ the space of cubic polynomials defined by 
\begin{equation}\label{eq:param}
P_{c,a}(z):=\frac{1}{3}z^3-\frac{c}{2}z^2+a^3~.
\end{equation}
It is a branched cover of the parameter space of cubic polynomials with marked critical points.
The critical points of $P_{c,a}$ are given by $c_0 := c$ and $c_1 := 0$.

For a fixed $(c,a)\in K^2$ the function $\frac13 \log^+|P_{c,a}(z)| - \log^+|z|$ is bounded on $\A^{1,\an}_K$ so that the sequence
$\frac1{3^n} \log^+|P^n_{c,a}(z)|$ converges uniformly to a continuous sub-harmonic function $g_{c,a}(z)$ that 
is called the Green function of $P_{c,a}$.

We shall write $g_0(c,a) := g_{c,a}(c_0)$, $g_1(c,a) := g_{c,a}(c_1)$, and 
\[G(c,a) := \max \{g_0(c,a), g_1(c,a)\}~.\]
\begin{proposition}\label{prop:growth Green}
The function $G(c,a)$ extends continuously to the analytification $\A^{2,\an}_K$, and
there exists a constant $C= C(K)>0$ such that 
\[\sup_{\A^{2,\an}_K} \left|G(c,a) - \log^+\max\{|a|,|c|\} \right| \le C~,\]
and this constant vanishes when the residual characteristic of $K$ is at least $5$.
\end{proposition}
\begin{proof}
A proof of this fact is given in~\cite[\S 4]{BH} (see also~\cite[\S 6]{favredujardin} for a more detailed proof) in the Archimedean case.
When the normed field is non-Archimedean, it is proved in~\cite[Proposition~2.5]{favregauthier} that the sequence $h_n:= \max\{\frac1{3^n} \log^+|P^n_{c,a}(c_0)|, \frac1{3^n} \log^+|P^n_{c,a}(c_1)|\}$ converges uniformly on bounded sets in $K^2$ to $G(c,a)$. Since $h_n$ extends continuously to $\A^{2,\an}_K$, it follows that 
$G$ too. The rest of the proposition also follows  from op. cit.
\end{proof}
\subsection{Expansion of the B\"ottcher coordinate} 
 For any cubic polynomial $P\in K[z]$, we let the \emph{B\"ottcher coordinate} of $P$ be the only formal power series $\varphi$ satisfying the equation
\begin{eqnarray}
\varphi\circ P(z)=\varphi(z)^3\label{eq:bott}
\end{eqnarray}
which is of the form 
\begin{eqnarray}
\varphi(z)=\om z+\alpha+\sum_{k\geq1}a_kz^{-k}~,\label{eq:phi-expand}
\end{eqnarray}
with $\alpha,a_k\in K$ for all $k\geq1$.
\begin{lemma}\label{lm:bottcher1}
Given any $(c,a)\in K\times K$, the B\"ottcher coordinate $\varphi_{c,a}(z)$ of the cubic polynomial $P_{c,a}:= \frac{z^3}{3}-\frac{c}2 z^2+a^3$ exists, is unique, and satisfies
\[\varphi_{c,a}(z)=\om\left(z-\frac{c}{2}\right)+\sum_{k\geq1}a_k(c,a)z^{-k},\]
where 
\begin{equation}\label{eq:deg-good}
a_k(c,a) \in \Z\left[\om,\frac{1}{2}\right][c,a]
\text{ with } \deg \left(a_k\right)=k+1~.
\end{equation}
Moreover the $2$-adic (resp. $3$-adic) norm of the coefficients of $a_k$ are bounded from above by $2^{k+1}$ (resp. $3^{k/2}$).
\end{lemma}

\begin{proof}
The defining equation~\eqref{eq:bott} reads as follows:
\begin{multline*}
\left(\om\left(z-\frac{c}{2}\right)+\sum_{k\geq1}a_k(c,a)z^{-k}\right)^3
= \\
\om\left(\frac{z^3}{3} -\frac{c}2 z^2+a^3 -\frac{c}{2}\right)+\sum_{k\geq1}\frac{3^k a_k(c,a)}{z^{3k} \,(1- \frac{3c}{2z}+\frac{3a^3}{z^3})^k}
\end{multline*}
An immediate check shows that terms in $z^3$ and $z^2$ are identical on both sides of the equation. 
Identifying terms in $z$ yields 
\[3 \om^3 (c^2/4) + 3 \om^2 a_1  = 0, \text{ so that }a_1 =  - \frac{\om}4 c^2,\]
whereas identifying constant terms, we get
\[3\om^2a_2 + 6 \om^2 (-c/2) a_1 + \om^3 (-c^3/8) = \om (a^3 - c/2)\]
hence
\[a_2 = - \frac{5\om}{24} c^3 
+ \frac1{3 \om} (a^3 - \frac{c}2)~.\]
This shows~\eqref{eq:deg-good} for $k=1,2$, since $\om^{-1}= 3\om$.

We now proceed by induction. Suppose~\eqref{eq:deg-good} has been proven for $k$.
Identifying terms in $z^{-(k-1)}$ in the equation above, we get
\begin{multline*}
3 \om^2 a_{k+1} - 3 c \om^2 a_k + \frac{3 \om^2}4 \, c^2 a_{k-1} +\\
+ \om \sum_{i+j=k} a_i a_j - \om \frac{c}2\sum_{i+j=k-1} a_i a_j 
+ \sum_{i+j+l = k+1} a_i a_j a_l 
= \\
\sum_{l\ge 1}3^l a_l \,  \left[\left(1+\frac{3c}{2z}+\frac{a^3}{z^3}\right)^{-l}\right]_{k+1- 3l}
\end{multline*}
where $ \left[\left(1+\frac{3c}{2z}+\frac{a^3}{z^3}\right)^{-l}\right]_{j}$ denotes the coefficient
in $z^{-j}$ of the expansion of $(1+\frac{3c}{2z}+\frac{a^3}{z^3})^{-l}$ in power of $z^{-1}$. 
Observe that this coefficient belongs to $\Z[\frac12] [c,a]$, has $2$-adic norm $\le 2^l$, and is a polynomial in $c,a$ of degree at most $j$.
It follows that the polynomial 
\[a_l(c,a) \,  \left[\left(1+\frac{3c}{2z}+\frac{a^3}{z^3}\right)^{-l}\right]_{k+1- 3l}\]
is of degree at most
$k+1 -3l + l+1 = k+2 - 2l < k+1$. The induction step is then easy to complete using again $\om^{-1}= 3\om$.
\end{proof}

\subsection{Extending the B\"ottcher coordinate} 

Recall that $G(c,a) = \max \{ g_0(c,a), g_1(c,a)\}$.
\begin{proposition}\label{prop:green is green}
There exists a constant $\rho = \rho(K)\ge 0$ such that the B\"ottcher coordinate of $P_{c,a}$ is converging in $\{z,\, \log |z| > \rho + G(c,a)\}$.

There exists another constant $\tau = \tau(K)\ge 0$ such that the map
$(c,a,z) \mapsto \varphi_{c,a}(z)$ extends as an analytic map on  the open set 
\[\{ (c,a,z)\in \A^{2,\textup{an}}_{K}\times\A^{1,\textup{an}}_{K}, \, g_{c,a}(z) > G(c,a)+ \tau\}~,\]
and $\varphi_{c,a}$ defines an analytic isomorphism
from $U_{c,a}:= \{g_{c,a}> G(c,a)+ \tau\}$  to $\A^{1,\textup{an}}_{K}\setminus\overline{\mathbb{D}(0,e^{G(c,a)+\tau})}$
satisfying the equation \eqref{eq:bott} on $U_{c,a}$. We have 
\begin{equation}\label{eq:green is green}
g_{c,a}(z)=\log|\varphi_{c,a}(z)|_K \ \text{ on } U_{c,a}~.
\end{equation}
Finally, $\tau = 0$ except if the residual characteristic of $K$ is equal to $2$ or $3$. 
\end{proposition}

We shall use the following lemma which follows easily from e.g.~\cite[Proposition~2.3]{favregauthier}. 
 \begin{lemma}\label{lem:classic-estim}
 There exists a  constant $\theta = \theta(K) \ge 0$ 
 \[\sup_{\A^{1,\an}_K} |g_{c,a}(z)- \log^+|z|| \le \theta~.\]
Moreover, $\theta$ is equal to $0$ except if the norm on $K$ is Archimedean or the residual characteristic of $K$ is equal to $2$ or $3$.
 \end{lemma}
\begin{proof}[Proof of Proposition~\ref{prop:green is green}]
Assume first that $K$ is Archimedean, and set $\tau =0$. In that case most of the statements are proved in~\cite{orsay1} (see also~\cite[\S 1]{BH}). 
In particular,  $\varphi_{c,a}(z)$ is analytic in a neighborhood of $\infty$ and extends to $U_{c,a}$ by invariance and 
defines an isomorphism between the claimed domains. It is moreover analytic in $c,a,z$. 

To estimate more precisely the radius of convergence of the power series~\eqref{eq:phi-expand}, we rely on \cite[\S 4]{BH} as formulated in
\cite[\S 6]{favredujardin}. First choose $C= C_K>0$ such that $G(c,a)  >   \log^+ \max \{|a| , |c| \} - C$. Then
$\log |z| > C + G(c,a)$ implies $|z- \frac{c}2| > \max \{1, |a| , |c| \} - |\frac{c}2|\ge \frac12  \max \{1, |a| , |c| \}$ hence
$\log |z- \frac{c}2| > G(c,a) - \log 2$, so that $g_{c,a}(z)> \log | z - \frac{c}2| - \log 4 > G(c,a)$,
and $\varphi$ converges in $\{z, \, \log |z| > G(c,a) + \rho\}$ with $\rho := C$ as required.

\smallskip

From now on, we assume that the norm on $K$ is non-Archimedean. 

\smallskip

When the  residual characteristic of $K$ is different from $2$ and $3$, then~\eqref{eq:deg-good} implies
$|a_k| \le \max\{1, |c|, |a|\}^{k+1}$ so that $\varphi$ converges for $|z| >  \max\{1, |c|, |a|\}$, and $\log|\varphi(z)| = \log|z|$.
Recall that we have $G(c,a) = \log \max\{1, |c|, |a|\}$ by Proposition~\ref{prop:growth Green} so that one can take $\rho =0$.
Pick any $z$ such that $g(z) > G(c,a)$, and observe that  $|P^n(z)| \to \infty$. Then we get 
\begin{equation}\label{eq:blop}
g_{c,a}(z) = \lim_{n\to\infty}\frac1{3^n} \log |P^n(z)| = \lim_{n\to\infty}\frac1{3^n}\log |\varphi(P^n(z))| = \log |\varphi(z)|=\log|z|~.
\end{equation}
In particular the set $\{g> G(c,a)\}$ is equal to $\A^{1,\textup{an}}_{K}\setminus\overline{\mathbb{D}(0,e^{G(c,a)}})$, and
$\varphi$ is an analytic map from that open set onto itself. It is an isomorphism since $\log|\varphi(z)| = \log|z|$ as soon as
$g(z) > G(c,a)$. The proposition is thus proved in this case with $\tau =0$.

\smallskip

In residual characteristic $2$, $|a_k| \le (2\, \max\{1, |c|, |a|\})^{k+1}$
whence $\varphi$ converges for $|z| >   2 \max\{1, |c|, |a|\}$, and as above $\log|\varphi| = \log |z|$ in that range. 
Recall that  $G(c,a) - \log^+\max \{|c|, |a|\} \ge C = C(K)$, so that $\log|z| > G(c,a) + \log 2 - C_K$ implies  $|z| >   2 \max\{1, |c|, |a|\}$,
which proves that the power series~\eqref{eq:phi-expand} converges for $\log |z| > G(c,a) + \rho$ with $\rho = \log 2 - C$.
Set $\tau := \rho + \theta$ where $\theta$ is the constant given by Lemma~\ref{lem:classic-estim}.
 Using $\log|\varphi(z)| = \log|z|$ as above, we get that $\varphi_{c,a}$ defines an analytic isomorphism
from $U_{c,a}:= \{g_{c,a}> G(c,a)+ \tau\}$  to $\A^{1,\textup{an}}_{K}\setminus\overline{\mathbb{D}(0,e^{G(c,a)+\tau})}$. 

\smallskip

In residual characteristic $3$,  $ |a_k| \le (3^{1/2} \, \max\{1, |c|, |a|\})^{k+1}$ 
whence $\varphi$ converges  $|z| >   3^{1/2} \, \max\{1, |c|, |a|\}$, and  $\log|\varphi| = \log |\om z|$ in that range.
Recall that  $G(c,a) - \log^+\max \{|c|, |a|\} \ge C = C(K)$, so that $\log|z| > G(c,a) +  \log 3^{1/2} - C_K$ implies  $|z| >   3^{1/2}  \max\{1, |c|, |a|\}$,
which proves that the power series~\eqref{eq:phi-expand} converges for $\log |z| > G(c,a) + \rho$ with $\rho = \log 3^{1/2} - C$.
We conclude the proof putting $\tau := \rho + \theta$ as before.
\end{proof}

\begin{remark}
It is possible to argue that $\tau =0$ also in residual characteristic $2$. Although we do not know the optimal constant $\tau$ in residual characteristic $3$, the B\"ottcher coordinate is likely not to induce an isomorphism from $\{g_{c,a}> G(c,a)\}$  to $\A^{1,\textup{an}}_{K}\setminus\overline{\mathbb{D}(0,e^{G(c,a)})}$.
\end{remark}

\section{Curves in $\poly_3$}\label{sec:curves in P3}
In this section we fix a number field $\KK$ containing a square-root $\om$ of $\frac13$ and take an irreducible curve $C$ in $\poly_3$ that is defined over $\KK$.
Our aim is to build suitable height functions on $C$ for which the distribution of points of small height can be described using Thuillier-Yuan's theorem.
Our main statement is Theorem~\ref{tm:bifurcationheights} below.

Recall that given any finite set $S$ of places of $\KK$ containing all Archimedean places, $\mathcal{O}_{\KK,S}$ denotes the ring of $S$-integers in $\KK$ that is of elements of $\KK$ of $v$-norm $\le 1$ for all $v \notin S$.
We also write $\KK_v$ for the completion of $\KK$ w.r.t. the $v$-adic norm.

\subsection{Adelic series}\label{sec:adelic series}

A formal power series $\sum_n a_n z^n$ is said to be \emph{adelic on} $\KK$ if its coefficients belong to $\mathcal{O}_{\KK,S}$ 
where $\KK$ is a number field, and $S$ a finite set of places on $\KK$; and for each place $v$ on $\KK$ 
the series has a positive radius of convergence $r_v:= \limsup_{n\to\infty} |a_n|_v^{-1/n} >0$.
Observe that $r_v =1$ for all but finitely many places.

\begin{lemma}\label{lem:adelic series}
Suppose $\alpha(t) = \sum_n a_n t^n$ is an adelic series with $a_0=0$ and $a_1\neq 0$. Then there exists an adelic series $\beta$ such that
$\beta \circ \alpha (t) =t$.
\end{lemma}

\begin{proof}
Suppose $a_n \in \mathcal{O}_{\KK, S}$  for all $n$, and
write $\beta(t) = \sum_n b_n t^n$. 
The equation  $\beta\circ \alpha (t) =t$ amounts to $b_0=0$, 
$b_1 = a_1^{-1}$, and the relations
\[b_n a_1^n +   \sum_{1\le k\le n-1} b_k \left[\left(\sum_{j\le n} a_j t^j\right)^k\right]_n =0~,\]
for any $n\ge2$ where $[\cdot]_n$ denotes the coefficient in $t^n$ of the power series inside the brackets.
It follows that  $b_n \in \mathcal{O}_{\KK, S'}$  for all $n$
where $S'$ is the union of $S$ and all places $v$ for which $|a_1|_v >1$.
The convergence of the series follows from Cauchy-Kowalewskaia's method of majorant series
or from the analytic implicit function theorem, see \cite{Chirka} and \cite[p. 73]{MR2179691}.
\end{proof}

\begin{lemma}\label{lem:adelic roots}
Pick $k\in \N^*$, and suppose $\alpha(t) = \sum_{n\ge k} a_n t^n$ is an adelic series with $a_k\neq 0$. Then there exists an adelic series $\beta$ such that
$\beta(t)^k = \alpha (t)$.
\end{lemma}

\begin{proof}
As in the previous proof,  suppose $a_n \in \mathcal{O}_{\KK, S}$  for all $n$, and write $\beta(t) = b_1 t + \sum_{n\ge 2} b_n t^n$. 
We get  $b_1^k = a_1$, and for all $n\ge2$
\[a_n=k b_1^{k-1} b_{n-k} + P_n(b_1, \ldots, b_{n-k-1})~,\]
where $P_n$ is a polynomial with integral coefficients. 
This time all coefficients $b_n$ belong to a finite extension of $\KK$ containing a fixed $k$-th root of $a_1$, and
$S'$ is the union of $S$ and all places $v$ such that $|k b_1^{k-1}|_v <1$.
The analyticity of the series is handled as in the previous proof.
\end{proof}

\begin{lemma}\label{lem:adelic pseudo root}
Pick $k\in \N^*$, and suppose $\alpha(t) = \sum_{n\ge k} a_n t^n$ is an adelic series with $a_k\neq 0$.
Then there exists an adelic series $\beta$ such that
$\alpha \circ \beta(t) = t^k$.
\end{lemma}

\begin{proof}
\smallskip
The equation $ \alpha \circ \beta(t) = t^k$ is equivalent to 
\[\left( 1+ \sum_{j\ge2} a_j t^{j-1}\right)^k \left(1 + \sum_{l\ge1} \alpha_l \left( t + \sum_{i\ge2} a_i t^i\right)^l\right)=1~.\]
Identifying terms of order $t^n$, one obtains
\[k a_{n+1} + \left[\left( 1+ \sum_{2\le j\le n} a_j t^{j-1}\right)^k \left(1 + \sum_{l\ge1} \alpha_l \left( t + \sum_{1\le i \le n} a_i t^i\right)^l\right)\right]_n= 0\]
which shows that $\beta$ is unique, has coefficients in $\mathcal{O}_{\mathbb{L},S'}$ where $S'$ contains $S$ and all places at which $|k|_v <1$.
The fact that $\beta$ is analytic at all places is a consequence of the inverse function theorem and the fact that the power series $t \mapsto (1+t)^{1/n} := 1+ \frac1n t + \frac{(1/n)(1/n-1)}2 t^2+ O(t^3)$ has a positive radius of convergence.
\end{proof}

We shall also deal with \emph{adelic series at infinity} which we define to be series of the form $\alpha (z)  =  \sum_{0\le k\le N} b_k z^k+ \sum_{k\ge1} \frac{a_k}{z^k}$ with $N\in \N$, $ b_k, a_k \in \mathcal{O}_{\KK,S}$ and
$\sum_{k\ge1} a_k t^k$ is an adelic series. Observe that this is equivalent to assume that $\alpha(t^{-1})^{-1}$ is an adelic series.

\subsection{Puiseux expansions}
We shall need the following facts on the Puiseux parameterizations of a curve defined over $\KK$.
These are probably well-known but we include a proof for the convenience of the reader. 
\begin{proposition}\label{lm:Puiseux}
Suppose $P\in \KK[x,y]$ is a polynomial such that 
$P(0,0) =0$ and $P(0,y)$ is not identically zero. Denote by $\mathsf{n}: \hat{D}\to D := \{P=0\}$
the normalization map, and pick any point $\mathfrak{c}\in \mathsf{n}^{-1}(0)\in\hat{D}$.

Then one can find a finite extension $\mathbb{L}$ of $\KK$, a finite set of places $S$ of $\mathbb{L}$, 
a positive  integer $n>0$, and an adelic series $\beta(t) \in \mathcal{O}_{\mathbb{L},S}[[t]]$ such that
\begin{enumerate}
\item
there is an isomorphism 
of  complete local rings $\widehat{\mathcal{O}_{\hat{D},\mathfrak{c}}}\simeq \mathbb{L}[[t]]$;
\item
the formal map $t\mapsto (t^n,\beta(t))$ parameterizes the branch $\mathfrak{c}$ in the sense that 
$x(\mathsf{n}(t))  = t^n$, and $y(\mathsf{n}(t)) = \beta(t)$.
\end{enumerate}
\end{proposition}
A branch of $D$ at the origin is by definition a point in $\mathsf{n}^{-1}(0)$. 
\begin{proof}
We first reduce the situation to the case $D$ is smooth at $0$. To do so
we blow-up  the origin $X_1 \to \A^2$ and let $D_1$ be the strict transform of $D$. Since $\hat{D}$ is normal
the map $\mathsf{n}$ lifts to a map $\mathsf{n}_1: \hat{D} \to D_1$, and we let
 $p_1$ be the image of $\mathfrak{c}$ in $D_1$. 
 
In the coordinates $(x,y)=(x', x'y')$ (or $(x'y', y')$) the point $p_1$ has coordinates $(0,y_1)$
 where $y_1$ is the solution of a polynomial with values in $\KK$ hence belongs to an algebraic extension of this field. 
We may  thus choose charts $(x,y) = (x', (x'+c)y')$ (or $(x' (y'+c), x')$) with $c\in \bar{\KK}$ such that 
$\mathfrak{c}$ is now a branch of $D_1 =\{P_1 =0\}$ at the origin, and $P_1 \in \bar{\KK}[x',y']$.  

We iterate this process of blowing-up to build a sequence of proper birational morphisms between 
smooth surfaces $X_{i+1} \to X_i$, $i=1, \ldots, N$ until we arrive at the following situation for $X:= X_N$:
the strict transform $C$ of $D$ by $\pi : X \to \A^2$ is smooth at a point $p\in \pi^{-1}(0)$ and 
intersects transversally the exceptional locus of $\pi$.

The normalization map $\mathsf{n}: \hat{D}\to D$ lifts to a map $\mathsf{m}: \hat{D}\to C$ and the image of $\mathfrak{c}$ by 
$\mathsf{m}$ is equal to $p$. Finally there exist coordinates $z,w$ centered at $p$ such that 
$(x,y) = \pi(z,w)= (A(z,w), B(z,w))$ with $A,B\in \bar{\KK}[z,w]$, the exceptional locus of $\pi$ contains $\{z=0\}$,
and $C = \{ R(z,w):= w - z a(z) - w Q(z,w) =0\}$ where $a\in \bar{\KK}[z]$, $Q\in \bar{\KK}[z,w]$ and $Q(0,0)=0$.

Fix an algebraic extension $\mathbb{L}$ of $\KK$ and $S$ finitely many places of $\mathbb{L}$ such that 
$A,B, R$ have their coefficients in $\mathcal{O}_{\mathbb{L},S}$.

We now look for a power series $\gamma (t) = \sum_{k\ge1} \gamma_k t^k$ such that 
$R(t,\gamma(t)) =0$. Its coefficients satisfy the relations
\[\gamma_k = [t^2a(t)]_k +   \left[ 
\left(\sum_{j=1}^{k-1} \gamma_j t^j\right)\, Q\left( t, \sum_{j=1}^{k-1} \gamma_j t^j\right)
\right]_k\]
which implies that $\gamma$ exists, is unique, and all its coefficients belongs to $\mathcal{O}_{\mathbb{L},S}$.
It follows from the analytic implicit function theorem, that $\gamma$ is also analytic as a power series in  $\mathbb{L}_v[[t]]$
for any place $v$.

\smallskip

Let us now consider the two power series $(\alpha(t), \delta(t)):= \pi(t, \gamma(t))$. They both belong to  $\mathcal{O}_{\mathbb{L},S}$,
are analytic at any place, and we have $P(\alpha(t),\delta(t))=0$. Since $P(0,y)$ is not identically zero, we may write
$\alpha (t) = t^n (a + \sum_{k\ge1} \alpha_k t^k)$ for some $n>0$ and $a\neq0$. Replacing $\mathbb{L}$ by a suitable finite extension, 
and $t$ by $a' t$ for a suitable $a'$ we may suppose that $a=1$ and $\alpha_k\in \mathcal{O}_{\mathbb{L},S}$ for all $k$.

By Lemma~\ref{lem:adelic pseudo root}, there exists an invertible  power series $\hat{a}(t) = t + \sum_{k\ge2} a_k t^k$ that is analytic at all places with coefficients
$a_k \in \mathcal{O}_{\mathbb{L},S}$ and such that 
$ \alpha \circ \hat{a} (t) = t^n$. Once this claim is proved one sets $\beta(t) := \delta \circ \hat{a}(t)$,  so that
$\pi ( \hat{a} (t), \gamma( \hat{a} (t))) = (t^n, \beta(t))$.

Since $\mathsf{m}$ is injective and maps the smooth point $\mathfrak{c}\in \hat{D}$ to the smooth point $p\in C$, 
it induces an isomorphism  of complete local rings $\widehat{\mathcal{O}_{C,p}}\simeq \widehat{\mathcal{O}_{\hat{D},\mathfrak{c}}}$. 
Observe that the complete local ring $\widehat{\mathcal{O}_{C,p}} = \mathbb{L}[[z,w]]/\langle R \rangle$ is  
isomorphic to $\mathbb{L}[[t]]$ by sending the class of a formal series $\Phi$ to $\Phi(t,\gamma(t))$).
Composing with the isomorphism  of $ \mathbb{L}[[t]]$ sending $t$ to $\hat{a}(t)$, we get 
an isomorphism  $\widehat{\mathcal{O}_{\hat{D},\mathfrak{c}}}\simeq\mathbb{L}[[t]]$ such that 
$(x(\mathsf{n}(t)), y(\mathsf{n}(t)))  = \pi(\mathsf{n}(t)) =  \pi (\hat{a}(t), \gamma(\hat{a}(t))) = (t^n,\beta(t))$
as required.
\end{proof}

\subsection{Branches at infinity of a curve in  $\textup{Poly}_3$}
Consider an irreducible affine curve $C\subset\poly_3$ defined over a number field $\KK$. 
We denote by $\overline{\poly_3}\simeq \mathbb{P}^2$ the natural compactification of $\poly_3\simeq \A^2$ using the affine coordinates $(c,a)$.
Let $\bar{C}$ be the Zariski closure of the curve $C$ in  $\overline{\poly_3}$, and $\mathsf{n}: \hat{C} \to \bar{C}$ be its normalization.
A \emph{branch at infinity} of $C$ is a point in $\hat{C}$ lying over
$\bar{C} \setminus C$.

\begin{proposition}\label{prop:defineL}
There exists a finite extension $\mathbb{L}$ of $\KK$ and a finite set of places $S$ such that the following holds.

For any branch $\mathfrak{c}$ of $C$ at infinity there is an isomorphism of complete local rings $\widehat{\mathcal{O}_{\hat{C},\mathfrak{c}}}\simeq \mathbb{L}[[t]]$
such that $c(\mathsf{n}(t)), a(\mathsf{n}(t))$ are adelic series at infinity.
\end{proposition}

\begin{proof}
Pick a branch at infinity $\mathfrak{c}$ of $C$. Let  $p_*$ be the image of $\mathfrak{c}$ in $\overline{\poly_3} \simeq \p^2$. It is given in homogeneous coordinates by $p_* = [c_*:a_*:0]$ and since $C$ is defined over $\KK$ we may assume $c_*, a_*$ are algebraic over $\KK$.
 To simplify the discussion we shall assume that $c_* =1$ so that $p_* = [1:a_*:0]$ (otherwise $p_* = [0:1:0]$ and the arguments are completely analoguous).
Let $d$ be the degree of a defining equation $P\in \KK[c,a]$ of $C$. Observe that $Q(\tau,\alpha) := \tau^d P(\frac1{\tau}, \frac{\alpha}{\tau}-a_*)$ is a polynomial
vanishing at $(0,0)$ such that $Q(0,\alpha)$ is not identically zero. Note that $\{ Q=0\}$ can be identified to an open Zariski subset of the completion of $\{P=0\}$ in $\overline{\poly_3}$, and $\mathfrak{c}$ with a branch of $\{ Q=0\}$ at the origin.

Apply Proposition~\ref{lm:Puiseux} to this branch $\mathfrak{c}$. We get a finite extension $\mathbb{L}$, a finite set of places $S$ of $\mathbb{L}$ containing all archimedean ones, a positive integer $n$, an isomorphism of complete local ring $\mathcal{O}_{\hat{C},\mathfrak{c}} \simeq \mathbb{L}[[t]]$, 
and a power series $\beta\in \mathcal{O}_{\mathbb{L},S}[[t]]$ that is analytic at all places such that
$\alpha(\mathsf{n}(t)) = \beta(t)$ and $\tau(\mathsf{n}(t)) = t^n$. It follows that
$c(\mathsf{n}(t)) = t^{-n}$, and
$a(\mathsf{n}(t)) = t^{-n} \beta(t) - a_*\in \mathcal{O}_{\mathbb{L},S}[[t]]$.
\end{proof}

\subsection{Estimates for the Green functions on a curve in  $\textup{Poly}_3$}

In this section, we fix an irreducible curve $C$ in $\poly_3$ defined over a number field $\mathbb{K}$ and let $\mathbb{L}$ be a finite extension of $\mathbb{K}$ for which Proposition~\ref{prop:defineL} applies. 
Fix a place $v$ of $\mathbb{L}$, and let $g_{0,v}(c,a)$ be the function $g_{0,v}$ evaluated at $c,a$ in the completion $\mathbb{L}_v$ of $\mathbb{L}$
with respect to the $v$-adic norm. 

By \cite{favredujardin} and \cite[Proposition 2.4]{favregauthier}, the function $g_{0,v}$ is the uniform limit on compact sets of $\frac1{3^n}\log^+ |P^n_{c,a}(c_0)|_v$. 
It follows that its lift to the normalization of $C$ is sub-harmonic (in the classical sense when $v$ is Archimedean and in the sense of Thuillier~\cite{thuillier} when 
$v$ is non-Archimedean).

\smallskip

To simplify notations, we also write 
$g_{0,v}(t) := g_{0,v}(c(\mathsf{n}(t)), a(\mathsf{n}(t)))$ where 
the adelic series at infinity $c(\mathsf{n}(t))$ and $a(\mathsf{n}(t))$ are given 
as  above.

\begin{proposition}\label{prop:branch}
For each branch $\mathfrak{c}$ of $C$ at infinity, one of the following two situations occur.
\begin{enumerate}
\item
For any place $v$ of $\mathbb{L}$, the function $g_{0,v}(t)$ extends as a locally bounded subharmonic function through $\mathfrak{c}$.
\item
There exist a finite set of place $S$ of $\LL$, and two constants $a(\mathfrak{c}) \in \Q_+^*$ and $b(\mathfrak{c}) \in \mathcal{O}_{\mathbb{L},S}$ such that 
$g_{0,v} (t) = a(\mathfrak{c}) \log|t|_v^{-1} + \log|b(\mathfrak{c})|_v + o(1)$ for any place $v$ on $\mathbb{L}$.
\end{enumerate}
\end{proposition}

\begin{remark}
This key result  is very similar to~\cite[Proposition~2.1]{BD}.
Ghioca and Ye have proved that $g_{0,v}(t)$ actually extends to a \emph{continuous} function at $t=0$ in case 1.
We also refer to~\cite[Proposition~3.1]{Demarco-stable} for a version of this result in the case of rational maps.  
\end{remark}

\begin{notation}
We endow the field $\LL((t))$ with the $t$-adic norm so that for any Laurent series $Q= \sum a_k t^k$ we have $|Q|_t :=e^{- \ord_t(Q)}$ with $\ord_t(Q) = \min\{k, \, a_k\neq 0\}$. The resulting valued field is complete and non-Archimedean.

In order to avoid confusion, we denote by $\mathsf{P}(z)\in \LL((t))[z]$ the cubic polynomial induced by the family $(P_{c(\mathsf{n}(t)),a(\mathsf{n}(t))})_t$. 
Observe that the critical points of $\mathsf{P}$ are given by $\mathsf{c}_0$ and $\mathsf{c}_1$
which correspond to the adelic series at infinity $0$ and $c(\mathsf{n}(t))$ respectively.
\end{notation}

\begin{proof}[Proof of Proposition~\ref{prop:branch}]
For each $q\in \N^*$, 
we set $e_q:= |\mathsf{P}^q(\mathsf{c}_0)|_t$, so that either the sequence $\{e_q\}_{q\in\N}$ is bounded (i.e. $\mathsf{c}_0$ belongs to the filled-in Julia set of $\mathsf{P}$)
or $e_q \to \infty$ (exponentially fast). 

\smallskip

Suppose we are in the former case, and consider the sequence of subharmonic functions $\frac1{3^q} \log^+ |P^q_t(0)|_v$ defined on a punctured disk
$\mathbb{D}^*_v$
centered at $0$ in $\A^{1,\an}_{\LL_v}$. Since $\frac1{3^q} \log^+ |P^q_t(0)|_v =  \frac{\log^+ e_q}{3^q} \log|t|_v^{-1} + O(1)$, the function 
\[h_q:= \frac1{3^q} \log^+ |P^q_t(0)|_v - \frac{\log^+ e_q}{3^q} \log|t|_v^{-1}\]
is subharmonic on $\mathbb{D}^*_v$ and locally bounded near $0$. It thus extends as a subharmonic function to $\mathbb{D}_v$ by the next lemma.
\begin{lemma}\label{lem:extend psh}
Any subharmonic function on $\mathbb{D}^*_v$ that is bounded from above in a neighborhood of the origin is the restriction of a subharmonic function on $\mathbb{D}_v$. 
\end{lemma}
\begin{proof}
In the Archimedean case, this follows from~\cite[Theorem~3.4.3]{MR2311920}. Let us explain how to adapt these arguments to the non-Archimedean setting.
Let $h$ be a subharmonic function on  $\mathbb{D}^*_v$, and suppose it is non-positive.
For each integer $n$ consider the function $h_{n}:= h + \frac1n  \log|t|$ with the convention $h_n(0) = -\infty$. Observe that 
$h_n$ is the pointwise decreasing limit of the sequence of subharmonic functions $\max \{ h_n , -A\}$ as $A\to\infty$, hence is subharmonic by~\cite[Proposition~3.1.9]{thuillier}.
Letting $n\to\infty$, we get an increasing sequence of subharmonic functions that is locally bounded and converging pointwise on $\mathbb{D}^*_v$ to $h$. 
The upper-semicontinuous regularization $h^*$ of $\lim_n h_n$ is thus subharmonic on $\mathbb{D}^*_v$ by~\cite[Proposition~3.1.9]{thuillier}  and extends $h$ as required. 
\end{proof}

 Since $\frac1{3^q} \log^+ |P^q_t(0)|_v$ converges uniformly on compact subsets in $\mathbb{D}^*_v$ to $g_{0,v}$,
 $h_q$ is uniformly bounded from above on its boundary, hence everywhere by the maximum principle.
 It follows from Hartog's theorem (see e.g.~\cite[Theorem 1.6.13]{MR1045639} in the Archimedean case, and~\cite[Proposition 2.18]{MR2578470} in the non-Archimedean case) that $h_q$ converges (in $L^1_\loc$ in the Archimedean case, and pointwise at any non-rigid point in the non-Archimedean case)
 to a subharmonic function, hence $g_{0,v}$ is subharmonic on $\mathbb{D}_v$. But $g_{0,v}$ is non-negative so that 
 (1) holds.   
 
\smallskip

Suppose that $e_q\to \infty$. Recall that $c(\mathsf{n}(t))$ and $a(\mathsf{n}(t))$ are adelic series at infinity that belong to $t^{-n} \mathcal{O}_{\mathbb{L},S}[[t]]$
for a suitable integer $n\ge1$. Write $\varphi_t:= \varphi_{P_{c(\mathsf{n}(t)), a(\mathsf{n}(t))}}$.
\begin{lemma}\label{lem:vlat}
There exists an integer $q\ge1$ such that for any place $v$ of $\LL$, there exists
$\epsilon>0$ such that $P^q_t(c_0)$ belongs to the domain of convergence of $\varphi_t$
for any $|t|_v<\epsilon$.
\end{lemma}
\begin{proof}
Indeed $P^q_t(c_0)$ is an adelic series at infinity having a pole of order $\log e_q$.
On the other hand, we have 
\begin{eqnarray*}
G(t) := G(c(\mathsf{n}(t)), a(\mathsf{n}(t))) & \le & \log \max \{|c(\mathsf{n}(t))|, |a(\mathsf{n}(t))|\}+ C\\
& \le & n \log|t|^{-1} + O(1)
\end{eqnarray*}
by Proposition~\ref{prop:growth Green}.
By assumption we may take $\log e_q$ to be as large as we want so that 
$\log|P^q_t(c_0)|_v - G(t) \to \infty$ for any fixed place $v$ when $|t|_v \to 0$.
We conclude by Proposition~\ref{prop:green is green}.
\end{proof}
Our objective is to estimate $\varphi_t(P^q_t(c_0))$.
Recall from Lemma~\ref{lm:bottcher1} that
\[\varphi_{c,a}(z)=\om\left(z-\frac{c}{2}\right)+\sum_{k\geq1}a_k(c,a)z^{-k},\]
with $a_k \in \Z\left[\om,\frac12\right][c,a]$ of degree $\le k+1$.
It follows that
\[\mathsf{a}_k:= a_k(c(\mathsf{n}(t)), a(\mathsf{n}(t)))\in t^{-n(k+1)}\,\mathcal{O}_{\mathbb{L},S}[[t]]~,\]
so that one can define
$$
\varphi_\mathsf{P}(z):=\varphi_{c(\mathsf{n}(t)), a(\mathsf{n}(t))} (z)
=
\om \left(z-\frac{c(\mathsf{n}(t))}{2}\right)+\sum_{k\geq1}\mathsf{a}_k z^{-k}
$$
as an element of the ring
$t^{-n}z\,\mathcal{O}_{\mathbb{L},S}((t))[[(t^{n}z)^{-1}]]$.

\smallskip

On the other hand, $P^q_{c,a}(c_0)$  is a  polynomial in $c,a$ of degree $\le 3^q$ 
with coefficients in $\Z\left[\frac12,\frac13\right]$ hence, if 
\[\mathsf{c}_0:=c_0(\mathsf{n}(t)) \text{ and }\mathsf{P}^q(z):=P^q_{c(\mathsf{n}(t)), a(\mathsf{n}(t))} (z)~,\]
we have $\mathsf{P}^q(\mathsf{c}_0)\in t^{-3^qn}\, \mathcal{O}_{\mathbb{L},S}[[t]]$, so that 
\begin{equation}\label{eq:pwr}
\frac{\mathsf{a}_k}{(\mathsf{P}^q(\mathsf{c}_0))^k} \in t^{3^q n k - n(k+1)}\,\mathcal{O}_{\mathbb{L},S}[[t]] \subset 
t^{nk}\,\mathcal{O}_{\mathbb{L},S}[[t]] 
~.\end{equation}
It follows that $\Theta:= \sum_{k\geq1}\frac{\mathsf{a}_k}{(\mathsf{P}^q(\mathsf{c}_0))^k}$ converges as a formal power series
and belongs to $t^n\,\mathcal{O}_{\mathbb{L},S}[[t]]$. Observe that 
Lemma~\ref{lem:vlat} shows that $\Theta$ is convergent at all places hence defines an adelic power series. 

\smallskip

Fix a place $v$ of $\LL$ and choose $|t|_v$ small enough. 
Then we get
\begin{eqnarray}\label{eq:expand bott}
\varphi_t(P^q_t(c_0)) 
=
\om\left(P_t^q(0)-\frac{c(\mathsf{n}(t))}{2}\right)+ \Theta(t) \\= \om\left(P_t^q(0)-\frac{c(\mathsf{n}(t))}{2}\right)+ o(1)~.\notag
\end{eqnarray}

By \eqref{eq:expand bott}, for $|t|_v$ small enough, one obtains
\begin{eqnarray*}
g_{0,v}(t) &=& \frac1{3^q} \log | \varphi_t (P^q_t(0))|_v \\
&=& 
\frac1{3^q} \log \left|\om \left(P^q_t(0) - \frac1{2t^n}\right)\right|_v+ o(1)\\
&=& 
\frac1{3^q} \log \left| \sum_{0\le k \le n_0} \frac{b_{k,0}}{t^k} \right|_v+ o(1)
= \\
&=&
\frac{n_0}{3^q} \log |t|_v^{-1} +  \log \left| \sum_{0\le k \le n_0} b_{k,0} t^{n_0-k} \right|_v+ o(1)\end{eqnarray*}
where $b_{k,0} \in \mathcal{O}_{\mathbb{L},S}$, and $b_{n_0,0} \neq 0$.
And the proof is complete with 
$a(\mathfrak{c}) := \frac{n_0}{3^q}$,
and  $b(\mathfrak{c}) = b_{n_0,0}$.
\end{proof}

\begin{proposition}\label{prop:adelic estim}
Fix any two positive integers $s:= (s_0, s_1)$, and for any place $v$ define
\begin{equation}\label{eq:def v-local} 
g_{s,v}(c,a):= \max \{s_0 g_{0,v}(c,a) , s_1 g_{1,v}(c,a)\}~.
\end{equation}
Then there exists an integer $q\ge 1$ such that 
\begin{equation}\label{eq:v-local is adelic}
g_{s,v} (c,a) = \frac1{3^q}\, \max \left\{ s_0  \log^+ |P^q_{c,a}(c_0)|, s_1  \log^+ |P^q_{c,a}(c_1)|\right\}
\end{equation}
for all but finitely many places. 
\end{proposition}

\begin{proof}
During the proof $S$ is a finite set of places on $\mathbb{L}$ that contains all Archimedean places and all places of residual characteristic $2$ and $3$.
Pick any $v\notin S$, and recall from~\cite[Proposition~2.5]{favregauthier} that $G_v(c,a) = \log^+ \max\{|c|_v, |a|_v\}$.

\smallskip

Suppose first that $g_{s,v}(c,a) =0$. Then $g_{c,a,v}(c_0) =g_{c,a,v}(c_1)=0$ and $G_v(c,a)=0$ so that 
$\frac1{3^q} \log^+|P^q_{c,a}(c_0)|_v = \frac1{3^q} \log^+|P^q_{c,a}(c_1)|_v=0$ for all $q$, and~\eqref{eq:def v-local}
holds in that case. 

\smallskip

Pick $q$ large enough such that $3^q > \max\{\frac{s_1}{s_0}, \frac{s_0}{s_1}\}$.
Suppose now that $0< g_{s,v}(c,a) = s_0 g_{0,v}(c,a)$ so that $s_0 g_{0,v}(c,a)\ge s_1 g_{1,v}(c,a)$. Then 
\begin{multline*}
g_{c,a,v}(P^q_{c,a}(c_0)) = 3^q g_{0,v}(c,a)\ge\\
\ge 3^q \min\left\{ \frac{s_1}{s_0},1\right\} \, \max \{g_{0,v}(c,a), g_{1,v}(c,a)\}> G_v(c,a)~.
\end{multline*}
By ~\eqref{eq:blop}, we get 
$$g_{s,v}(c,a) = s_0 g_{0,v}(c,a) = \frac{s_0}{3^q} g_{0,v}(P^q_{c,a}(c_0))=  \frac{s_0}{3^q} \log^+|P^q_{c,a}(c_0)|_v~.$$

\smallskip

Now observe that either $P_{c,a}^q(c_1)$ falls into the domain of definition of $\varphi_{c,a}$ i.e. $\log|P_{c,a}^q(c_1)|_v > G_v(c,a)$ 
and  $g_{1,v}(c,a) = \frac{1}{3^q} \log^+|P^q_{c,a}(c_1)|_v$, so that 
\begin{eqnarray*}
g_{s,v} (c,a) & = &
\max\{ s_0 g_{0,v}(c,a), s_1 g_{1,v}(c,a)\}\\
& = & \frac1{3^q}\, \max \left\{ s_0  \log^+ |P^q_{c,a}(c_0)|_v, s_1  \log^+ |P^q_{c,a}(c_1)|_v\right\}~,
\end{eqnarray*}
as required.
Or  we have 
\[\frac{s_1}{3^q} \log^+|P^q_{c,a}(c_1)|_v\le \frac{s_1}{3^q} \log^+\max\{|a|_v, |c|_v\}\le s_0 g_{0,v}(c,a)~,\]
and again~\eqref{eq:v-local is adelic}
holds.

We complete the proof by arguing in the same way when $g_{s,v}(c,a) = s_1 g_{1,v}(c,a)$.
\end{proof}

\subsection{Adelic  semi-positive metrics on curves in  $\textup{Poly}_3$}
We fix a number field $\mathbb{L}$ and a finite set $S$ of places of this field
that contains all Archimedean places and all places of residual characteristic $2$ and $3$. We also assume that
Propositions~\ref{prop:defineL},~\ref{prop:branch} and~\ref{prop:adelic estim} are all valid for these choices.

\smallskip

Fix any pair of positive integers $s_0, s_1\in \N^*$.
For each place $v$, introduce the function 
\[g_{s,v}(c,a):=\max\left\{s_0\cdot g_{0,v}(c,a),s_1\cdot g_{1,v}(c,a)\right\}~,\]
as in the previous section.

Pick a branch at infinity $\mathfrak{c}$ and choose parameterizations such that Proposition~\ref{prop:branch} is valid for $g_{0,v}$ and $g_{1,v}$. 
Observe that 
\[G_v(t) = \max \{g_{0,v}(t), g_{1,v}(t)\} \to \infty\]
as $t\to 0$ by Proposition~\ref{prop:growth Green} so that either
$g_{0,v}$ or $g_{1,v}$ tends to infinity near $t=0$. 
Since $n_0$ and $n_1$ are both positive, we conclude to the existence of $a(\mathfrak{c}) \in \Q^*_+$ and $b(\mathfrak{c})\in \mathcal{O}_{\mathbb{L},S}$ such that
\begin{equation}\label{eq:continuous metric}
g_{s,v} (t) = a(\mathfrak{c}) \log|t|_v^{-1} + \log|b(\mathfrak{c})|_v + o(1)~.
\end{equation}
We replace the integers $s_0,s_1$ by suitable multiples such that the constants $a(\mathfrak{c})$ become integral (for all branch $\mathfrak{c}$), and we
introduce the divisor $\mathsf{D} := \sum a(\mathfrak{c}) \, [\mathfrak{c}]$ on $\hat{C}$ where the sum is taken over all branches at infinity of $C$.

\smallskip

Pick a place $v$, an open subset $U\subset \hat{C}^{\an,v}$ and a section $\sigma$ of the line bundle $\mathcal{O}_{\hat{C}}(\mathsf{D})$ over $U$. 
By definition $\sigma$  is a meromorphic function on $U$ whose divisor of poles and zeroes satisfies $\ndiv (\sigma) + \mathsf{D} \ge 0$.
We set $|\sigma|_{s,v} :=  |\sigma|_ve^{-g_{s,v}}$.

Recall the notion of semi-positive metrics in the sense of Zhang from~\cite[\S 1.2.8 \& 1.3.7]{ACL2}.
We are now in position to prove
\begin{lemma}
The metrization $|\cdot|_{s,v}$  on the line bundle  $\mathcal{O}_{\hat{C}}(\mathsf{D})$ is continuous and semi-positive for any place $v$.  The collection of metrizations $\{|\cdot|_{s,v}\}_v$ is adelic.
\end{lemma}

\begin{proof}
For any place $v$, and for any local section $\sigma$ of the line bundle  $\mathcal{O}_{\hat{C}}(\mathsf{D})$ the function $|\sigma|_{s,v}$ is continuous by~\eqref{eq:continuous metric}, therefore the metrization $|\cdot|_{s,v}$ is continuous.

\smallskip

Since $g_{s,v}$ is subharmonic on $C^{v,\an}$, for any local section  $\sigma$ the function 
 $-\log |\sigma|_{s,v}$ is subharmonic on $C^{v,\an}$. As it extends continuously to $ \hat{C}^{\an,v}$, 
Lemma \ref{lem:extend psh} implies that $-\log |\sigma|_{s,v}$ is subharmonic on $ \hat{C}^{\an,v}$.
When $v$ is Archimedean, the metrization is thus semi-positive by definition. When 
$v$ is non-Archimedean the metrization is semi-positive by the next lemma. 

\smallskip

Finally the collection of metrizations $\{|\cdot|_{s,v}\}_v$ is adelic thanks to Proposition~\ref{prop:adelic estim}
and \cite[\S 2.3]{favregauthier}.
\end{proof}

\begin{lemma}
Suppose $L\to X$ is an ample line bundle on a smooth projective curve over a metrized field.
Let $|\cdot|$ be any continuous metrization on $L$ that is subharmonic in the sense that for any local section $\sigma$ the function $-\log |\sigma|$ is subharmonic. 

Then the metrization $|\cdot|$ is semi-positive in the sense of Zhang.
\end{lemma}

\begin{remark}
This result holds true in arbitrary dimension  over any field of the form $L((t))$ endowed with the $t$-adic norm where $L$ has characteristic zero by~\cite{MR3419957}. 
\end{remark}

\begin{proof}
One needs to show that $|\cdot|$ is the uniform limit of a sequence of semi-positive model metrics.
By~\cite[Th\'eor\`eme 3.4.15]{thuillier}, there exists an inductive set $I$ and an \emph{increasing} family of semi-positive model metrizations $|\cdot|_i$ that are pointwise converging to $|\cdot|$.

Since $|\cdot|$ is continuous and $X$ is compact, Dini's theorem applies and shows that the convergence is uniform. 

Observe that the notion of semi-positivity used in op. cit. for model metrics coincides with the one in~\cite{zhang} as explained in \cite[\S 4.3.2]{thuillier}.
\end{proof}

We have thus obtained

\begin{theorem}\label{tm:bifurcationheights}
Pick any positive integers $s_0,s_1>0$. 
Then there exists a positive integer $n\ge1$, and a non-zero effective and integral divisor $\mathsf{D}$ on $\hat C$
such that the collection of subharmonic functions 
\[g_{s,v}(c,a):=\max\left\{ns_0\cdot g_{0,v}(c,a), ns_1\cdot g_{1,v}(c,a)\right\}~, \ (c,a)\in C^{v,\an}\]
induces a semi-positive adelic metrization on the line bundle $\mathcal{O}_{\hat C}(\mathsf{D})$.
\end{theorem}

\begin{remark}
Since the metrization $|\cdot|_{s,v}$  is continuous its curvature form (see~\cite{ACL2}) does not charge any point, and is given by the pull-back by $\mathsf{n}$ of the positive measure $\Delta g_{s,v}$ restricted to the set of regular points on $C$, see~\cite[\S 3.4.3]{thuillier}. To simplify notations we shall simply write this curvature form as $\Delta g_{s,v}$.
\end{remark}

\begin{remark}
The line bundle $\mathcal{O}_{\hat C}(\mathsf{D})$ is defined over the same number field as $C$.
\end{remark}
\begin{remark}
 It is likely that $g_{s,v}$ defines a semi-positive adelic metrization on an ample line bundle over a suitable compactification of $\poly_3$, but this seems quite delicate to prove for arbitrary $s=(s_0,s_1)\in (\mathbb{N}^*)^2$.
\end{remark}

\section{Green functions on special curves}\label{sec:small points}

This section is devoted to the proof of Theorem~\ref{tm:samegreen} below.
If $\KK$ is a number field, and $v$ a place of $\KK$, recall the definition of $\tau_v= \tau(\KK_v)$ from Proposition~\ref{prop:green is green}, and that $\tau_v =0$
if the residual characteristic of $\KK$ is larger than $5$.
\begin{theorem}\label{tm:samegreen}
Let $C$ be an irreducible curve in the space $\poly_3$ of complex cubic polynomials parameterized as in~\eqref{eq:param}.
Suppose that $C$ contains infinitely many post-critically finite parameters and that neither $c_0$ nor $c_1$ is 
persistently pre-periodic. 
Then the following holds.
\begin{enumerate}
\item The curve $C$ is defined over a number field $\KK$ and there exist positive integers 
$s_0, s_1$ such that for any place $v$ of  $\KK$
\[s_0\, g_{0,v} (c,a)= s_1\, g_{1,v} (c,a) \text{ for all } (c,a) \in C^{v,\an} ~.\]
\item
For any branch $\mathfrak{c}$ of $C$ at infinity, there exists an integer $q\geq1$ and  a root of unity $\zeta$ such that 
for any place $v$ of $\KK$, one has
\begin{equation}\label{eq:root}
\left(\varphi_{c,a}(P_{c,a}^{q}(c_0))\right)^{s_0}=\zeta \cdot \left(\varphi_{c,a}(P_{c,a}^{q}(c_1))\right)^{s_1}
\end{equation}
on the connected component of $\{g_{0,v}>\tau_v/s_0\}=\{g_{1,v}>\tau_v/s_1\}$ in $C^{v,\an}$ whose closure in $\hat{C}$ contains $\mathfrak{c}$.
\end{enumerate}
\end{theorem}
A remark is in order about the second assertion of the theorem.

\begin{remark}
We shall prove that for any parameter on the connected component $\{g_{0,v}>\tau_v/s_0\}=\{g_{1,v}>\tau_v/s_1\}$ in $C^{v,\an}$ whose closure in $\hat{C}$ contains $\mathfrak{c}$, 
the two points $P_{c,a}^{q}(c_0)$ and $P_{c,a}^{q}(c_1)$ belong to the domain of definition of the B\"ottcher coordinate $\varphi_{c,a}$ for $q$ large so that
\eqref{eq:root} is consistent. Note that we shall prove that \eqref{eq:root} holds as an equality of adelic series at infinity.
\end{remark}

\subsection{Green functions are proportional}\label{sec:propor}
The set of post-critically finite polynomials is a countable union of varieties
\[V_{n,m}:= \{P^{n_0+m_0}_{c,a}(c_0) =   P^{n_0}_{c,a}(c_0)\} \cap \{P^{n_1+m_1}_{c,a}(c_1) =   P^{n_1}_{c,a}(c_1)\}\]
with $n=(n_0, n_1)\in \N^2$ and $m=(m_0, m_1)\in (\N^*)^2$, and each of which is cut out by two polynomial equations with coefficients in $\Z\left[ \frac12, \frac13\right]$. 
Since $V_{n,m}(\C)$ are all contained in a fixed compact set by~\cite{BH} (see also~\cite[Proposition 6.2]{favredujardin}), it is a finite set, hence all its solutions are defined over a number field. 

It follows that $C$ is an irreducible curve containing infinitely many algebraic points $(c_n,a_n)$. Let $Q \in \C [c,a]$ be a defining equation for $C$ with at least one coefficient equal to $1$ and pick $\sigma$
an element of the Galois group of $\C$ over the algebraic closure of $\Q$. Then $Q \circ \sigma$ vanishes also on $\{(c_n,a_n)\}$ hence everywhere on $C$,
and therefore $Q\circ \sigma = \lambda Q$ for some $\lambda \in \C^*$. Since one coefficient of $Q$ is $1$, we get $\lambda =1$ and $Q\in \KK[c,a]$ for a
number  field $\KK$.

\smallskip

Recall that we denote by $\mathsf{n}: \hat{C}\to \bar{C}$ the normalization of the completion $\bar{C}$ of $C$ in $\overline{\poly_3}\simeq\p^2$. 
Pick any pair of positive integers $s=(s_0,s_1)$ and scale them such that Theorem~\ref{tm:bifurcationheights} applies. This gives us a non-zero effective divisor
$\mathsf{D}_s$ supported on $\hat{C}\setminus \mathsf{n}^{-1}(C)$.
Replacing $s$ by a suitable multiple, we may suppose that it is very ample and pick a rational function $\phi$ on $\hat{C}$
whose divisor of poles and zeroes is greater or equal to $-\mathsf{D}_s$. Observe in particular that $\phi$ is a  regular function  on $\mathsf{n}^{-1}(C)$ that vanishes at finitely many points.

Consider the height $h_s$ induced by the  semi-positive adelic metrics given by $g_{s,v}$, see Theorem~\ref{tm:bifurcationheights}. If $(c,a)$ is a point in $\mathsf{n}^{-1}(C)$ that is defined over a finite extension $\KK$, 
denote by $\mathsf{O}(c,a)$ its orbit under the action of the absolute Galois group of $\KK$, and by $\deg(c,a)$ the cardinality of this orbit. 
Fix a rational function $\phi$ as above that is not vanishing at $(c,a)$ (this exists since $-\mathsf{D}_s$ is very ample).
Let $M_\KK$ be the set of places of $\KK$.
By~\cite[\S 3.1.3]{ACL2}, since $\phi(c,a)\neq0$ we have
\begin{eqnarray*}
h_s(c,a) 
&= &
\frac1{\deg(c,a)} \sum_{\mathsf{O}(c,a)} \sum_{v\in M_\KK}  - \log |\phi|_{s,v}(c',a')
\\
&=&
\frac1{\deg(c,a)} \sum_{\mathsf{O}(c,a)} \sum_{v\in M_\KK}  (g_{s,v} -  \log |\phi|_v) (c',a')
\\
&=&
\frac1{\deg(c,a)} \sum_{\mathsf{O}(c,a)} \sum_{v\in M_\KK}  g_{s,v} (c',a')\ge 0
\end{eqnarray*}
where the last equality follows from the product formula.

\smallskip

We now estimate the total height of the curve $\hat{C}$ using~\cite[(1.2.6) \& (1.3.10)]{ACL2}.
Choose any two meromorphic functions $\phi_0, \phi_1$  such that 
$\mathrm{div}(\phi_0)+ \mathsf{D}_s$ and $\mathrm{div}(\phi_1)+ \mathsf{D}_s$ are both effective with disjoint support included in $\mathsf{n}^{-1}(C)$.  Let $\sigma_0$ and $\sigma_1$ be the associated sections of $\mathcal{O}_{\hat{C}}(\mathsf{D}_s)$. Let $\sum n_i [c_i,a_i]$ be the divisor of zeroes of $\sigma_0$, and
$\sum n'_j [c'_j,a'_j]$ be the divisor of zeroes of $\sigma_1$.
Then 
\begin{eqnarray*}
h_s(\hat{C}) 
& = &
\sum_{v\in M_\KK} (\hdiv(\sigma_0)\cdot \hdiv(\sigma_1) | \hat{C})_v
\\
& = &
\sum_i n_i h_s(c_i,a_i) - 
\sum_{v\in M_\KK} \int_{\hat{C}} \log| \sigma_0|_{s,v}\,  \Delta g_{s,v}
\\
& = &
\sum_{v\in M_\KK} \int_{\hat{C}}  g_{s,v}\,  \Delta g_{s,v} \ge0~,
\end{eqnarray*}
where the third equality follows from Poincar\'e-Lelong formula and writing 
$\log| \sigma_0|_{s,v} = \log |\phi|_v - g_{s,v}$ with $\phi\in \KK(C)$ defining the section $\sigma_0$.

\smallskip

The formula for the height of a closed point implies that for all  post-critically finite polynomials $P_{c_n,a_n}$ we have
$h_s(c_n,a_n) =0$.  Since PCF polynomials are Zariski dense in $C$, the essential minimum of $h_s$ is non-positive. 
By the arithmetic Hilbert-Samuel theorem (see~\cite[Th\'eor\`eme~4.3.6]{thuillier},~\cite[Proposition~3.3.3]{MR1810122}, or ~\cite[Theorem 5.2]{zhang}), we get $h_s(\hat{C}) =0$
hence we may apply Thuillier-Yuan's theorem (see~\cite{thuillier,yuan} and~\cite[Th\'eor\`eme~4.2]{ACL}). It follows
that the sequence of probability measures $\mu_{n,v}$ that are equidistributed on $\mathsf{O}(c_n,a_n)$
in $\hat{C}^{v,\an}$ converges to a probability measure $\mu_{\infty,v}$ that is proportional to $\Delta g_{s,v}$.
We may thus write $\mu_{\infty,v} = w(s) \, \Delta g_{s,v}$ where $w(s) \in \R^*_+$  is equal to the
inverse of the mass of $\Delta g_{s,v}$, i.e. to $\deg(\mathsf{D}_s)^{-1}$.

We now observe that $g_{s,v}$ is homogeneous in $s$ (i.e. $g_{\tau s,v}=\tau g_{s,v}$ for any $\tau\in\R^*_+$), and continuous with respect to this parameter. 
It follows that $w(s)$ is also continuous on $(\R^*_+)^2$, and $\mu_{\infty,v} = w(s) \, \Delta g_{s,v}$
for all $s\in(\R^*_+)^2$.

\medskip

From now on we fix an Archimedean place $v$. We shall treat the non-Archimedean case latter. 
We work in $\mathsf{n}^{-1}(C^{v,\an})$ which is the complement of finitely many points in the analytification 
of the smooth projective curve $\hat{C}^{v,\an}$. To simplify notation we write $g_{0,v}, g_{1,v}$ instead of $g_{0,v}\circ \mathsf{n}, g_{1,v}\circ \mathsf{n}$.
Recall that by~\cite[Theorem 2.5]{favredujardin} (see also~\cite[Theorem 2.2]{McMullen4} or~\cite[Theorem 1.1]{Demarco-stable}) the equality $g_{0,v} =0$ on 
$\mathsf{n}^{-1}(C^{v,\an})$ implies
$c_0$ to be persistently pre-periodic. Since we assumed that both $c_0$ and $c_1$ are not persistently pre-periodic, the functions $g_{0,v}$ and $g_{1,v}$
are not identically zero on $\mathsf{n}^{-1}(C^{v,\an})$.

Recall also that $g_{0,v}$ is harmonic where it is positive and that the support of $\Delta g_{0,v}$ is exactly the boundary of $\{g_{0,v} =0\}$ (see e.g.~\cite[Proposition 6.7]{favredujardin}).
In particular $\Delta g_{0,v}$ is a non-zero positive measure, and its mass is finite by Proposition~\ref{prop:branch}. Observe now that $g_{s,v} \to g_{0,v}$ uniformly on compact sets
when $s$ tends to $(1,0)$, hence  $\Delta g_{s,v} \to \Delta g_{0,v}$ and $\Delta g_{0,v} = t_0 \mu_{\infty,v}$ for some positive $t_0$.  In the same way, we get $\Delta g_{s,v} \to \Delta g_{1,v}$ as $s \to (0,1)$ which implies that the three positive measures $\mu_{\infty,v}$,
$\Delta g_{0,v}$ and $\Delta g_{1,v}$ are proportional.
We may thus find $s_0, s_1 >0$ such that the function  $H_v:= s_0 g_{0,v} - s_1 g_{1,v}$ is harmonic on $\mathsf{n}^{-1}(C^{v,\an})$. 

\smallskip

Recall from~\cite{McMullen} that the bifurcation locus of the family $P_{c,a}$ parameterized by $(c,a)\in \mathsf{n}^{-1}(C^{v,\an})$ is defined as the 
set where either $c_0$ or $c_1$ is unstable (or active in the terminology of~\cite{favredujardin}). It follows from~\cite{favredujardin} that the bifurcation locus is equal to the union
of the support of $\Delta g_{0,v}$ and $\Delta g_{1,v}$, hence to the support of $\mu_{\infty,v}$.

Suppose now that $H_v$ is not identically zero. Then this support is included in the locus $\{H_v =0\}$ which is real-analytic.
This is impossible by McMullen's universality theorem, since
the Hausdorff dimension of the bifurcation locus of any one-dimensional analytic family is equal to $2$, see~\cite[Corollary~1.6]{McMullen3}.

We have proved that  $ s_0 g_{0,v} = s_1 g_{1,v}$ on $\mathsf{n}^{-1}(C^{v,\an})$ hence on $C^{v,\an}$ for some positive real numbers $s_0, s_1>0$.

\smallskip

Since $g_{0,v}$ and $g_{1,v}$ are proportional, and $G_v = \max\{g_{0,v}, g_{1,v}\}$ is proper on $C^{v,\an}$, 
it follows that $g_{0,v}$ is unbounded near any branch at infinity.
By Proposition~\ref{prop:branch},  $g_{0,v}$ admits an expansion
 of the form $g_{0,v}(t) = a(\mathfrak{c}) \log|t|^{-1} + O(1)$ with $a(\mathfrak{c})\in \Q^*_+$ on the branch $\mathfrak{c}$
 hence is locally superharmonic on that branch. 
 
It follows that  $\Delta g_{0,v}$ is a signed measure in $\hat{C}^{\an,v}$ whose negative part is a divisor
$\mathsf{D}_0$ with positive \emph{rational} coefficients at any point of $\hat{C} \setminus \mathsf{n}^{-1}(C)$.
The same being true for $\Delta g_{1,v}$, we conclude to the equality of divisors $s_0 \mathsf{D}_0 = s_1 \mathsf{D}_1$.
This implies that $s_0/s_1$ is rational, and we can assume $s_0$ and $s_1$ to be integers.
This ends the proof of the first statement in the case the place is Archimedean.

\smallskip

Assume now that $v$ is non-Archimedean. One cannot copy the proof we gave in the Archimedean setting since
we used the fact that $c_0$ is not persistently pre-periodic iff $\Delta g_{0,v} = 0$, and  
McMullen's universality theorem, two facts that  are valid only over $\C$.

Instead we apply Proposition~\ref{prop:branch}. For each $s'=(s'_0,s'_1)$ the function $g_{s',v}$ extends near any branch $\mathfrak{c}$ at infinity
as an upper-semicontinuous function $\widehat{g_{s',v}}$  whose Laplacian puts some non-positive mass at $\mathfrak{c}$. 
When $s'_0, s'_1 \neq 0$ then $g_{s',v}$ defines a positive continuous metric on $\mathcal{O}_{\hat{C}}(\mathsf{D}_s)$ hence
$\Delta \widehat{g_{s',v}} \{\mathfrak{c}\} = - \ord_{\mathfrak{c}}(\mathsf{D}_s)<0$.
This mass is in particular independent on the place. 
We get that 
\[- \Delta \widehat{g_{0,v}} \{\mathfrak{c}\} \ge \lim_{s\to (1,0)} 
- \Delta \widehat{g_{s,v}} \{\mathfrak{c}\} = \ord_{\mathfrak{c}}(\mathsf{D}_0)>0~.\]
We infer that the mass of $\Delta g_{0,v}$ is equal to the degree of $\mathsf{D}_0$ hence
is non-zero.

We may now argue as in the Archimedean case, and prove that $\Delta g_{0,v}$ and $\Delta g_{1,v}$
are proportional. The coefficient of proportionality is the only $t>0$ such that 
$\mathsf{D}_0 = t \mathsf{D}_1$ hence $t= s_0/s_1$. Then $H_v:= s_0g_{0,v} - s_1 g_{1,v}$ is harmonic on 
$C$ and bounded near any branch at infinity by Proposition~\ref{prop:branch}, hence defines a harmonic function
on the compact curve $\hat{C}^{\an,v}$. It follows $H_v$ is a constant (in the non-Archimedean case by~\cite[Proposition 2.3.2]{thuillier}) which is necessarily zero since it is zero at all post-critically finite parameters.

We have completed the proof of Theorem~\ref{tm:samegreen} (1). 
\smallskip

We mention here the following result that follows from the previous argument.

\begin{corollary}\label{cor:distrib}
Let $C$ be an irreducible curve in $\poly_3$ defined over a number field $\KK$ that
contains infinitely many post-critically finite parameters and that neither $c_0$ nor $c_1$ is 
persistently pre-periodic. Pick an Archimedean place $v$.

Pick any sequence $X_n\subset C(\bar{\KK})$ of  Galois-invariant finite sets of postcritically finite parameters such that $X_n\neq X_m$ for $m\neq n$.  Let $\mu_n$ be the measure equidistributed on $X_n\subset C^{v,\an}$. 

Then the sequence $\mu_n$ converges weakly to (a multiple of) $T_\bif\wedge[C]$ as $n\to\infty$.
\end{corollary}

Recall that $T_\bif$ is defined as the $dd^c$ of the plurisubharmonic function $g_0+g_1$, and $[C]$ is the current of integration over the analytic curve
$C^{v,\an}$.

\begin{proof}
Let $s_0,s_1>0$ be given by Theorem~\ref{tm:samegreen}. As seen above, the sequence $\mu_n$ converges weakly towards $dd^c\max\{s_0\cdot g_0,s_1\cdot g_1\}=s_0\cdot dd^cg_0$ on $ C^{v,\an}$. It thus only remains to prove that $dd^c(g_0|_{C^{v,\an}})=\kappa\cdot T_\bif\wedge[C]$ for some $\kappa>0$. Recall that $T_\bif=dd^c(g_0+g_1)$. By Theorem~\ref{tm:samegreen}, on $C$,
\[g_0+g_1=g_0+\frac{s_1}{s_0}\cdot g_0=\left(1+\frac{s_0}{s_1}\right)g_0.\]
Let $\kappa:=1+\frac{s_0}{s_1}$. We thus have $dd^c(g_0|_{ C^{v,\an}})=\kappa^{-1} \cdot dd^c((g_0+g_1)|_{ C^{v,\an}})$. Finally, since $g_0+g_1$ is continuous, we have $T_\bif\wedge[C]=dd^c((g_0+g_1)|_{ C^{v,\an}})$, which ends the proof.
\end{proof}

\subsection{Values of the B\"ottcher coordinates at critical points are proportional near infinity}
In this section, we prove Theorem~\ref{tm:samegreen} (2). 

Let us fix a branch at infinity $\mathfrak{c}$ of an irreducible curve $C$ containing infinitely many PCF polynomials, 
and an isomorphism of complete local rings $\mathcal{O}_{\hat{C},\mathfrak{c}} \simeq \mathbb{L}[[t]]$, 
such that $c(\mathsf{n}(t)) = t^{-n}$, and $a(\mathsf{n}(t))\in \mathcal{O}_{\mathbb{L},S}((t))$ is an adelic series.
Write $P_t = P_{c(\mathsf{n}(t)),a(\mathsf{n}(t))}$, and $\varphi_t= \varphi_{P_t}$.

By Lemma~\ref{lem:vlat}
there exists an integer $q\ge1$ large enough such that 
$P_t^q(c_0)$ and $P_t^q(c_1)$ both lie in the domain of convergence of the B\"ottcher coordinate $\varphi_t$ for $t$ small enough, and
\eqref{eq:expand bott} holds, i.e.
\begin{equation*}
\varphi_t(P_t^q(c_\varepsilon)) = 
\om\left(P_t^q(c_\varepsilon)-\frac{c(\mathsf{n}(t))}{2}\right)+ \Theta(t)
\end{equation*}
where $\Theta$ is an adelic series vanishing at $0$.

We now fix a place $v$ and compute using Proposition~\ref{prop:green is green} for $|t|_v \ll 1$. We get
\begin{eqnarray*}
 \frac{|\varphi_t\left(P_0^q(t)\right)|_v^{s_0}}{|\varphi_t\left(P_1^q(t)\right)|_v^{s_1}} & = & \frac{\exp\left(s_0\cdot g_{c(\mathsf{n}(t)), a(\mathsf{n}(t))}(P_0^q(t))\right)}{\exp\left(s_1\cdot g_{c(\mathsf{n}(t)), a(\mathsf{n}(t))}(P_1^q(t))\right)}\\
& = & \frac{\exp\left(3^q s_0\cdot g_{0,v}(c(\mathsf{n}(t)), a(\mathsf{n}(t)))\right)}{\exp\left(3^q s_1\cdot g_{1,v}(c(\mathsf{n}(t)), a(\mathsf{n}(t))\right)}=1~, \hspace{1cm} (\star)
\end{eqnarray*}
where the last equality follows from Theorem~\ref{tm:samegreen} (1). 

Applying $(\star)$  in the case of an Archimedean place, we see that the complex analytic map 
\[t\longmapsto\frac{\left(\varphi_t\left(P_0^q(t)\right)\right)^{s_0}}{\left(\varphi_t\left(P_1^q(t)\right)\right)^{s_1}}\]
has a modulus constant equal to $1$, hence is a constant, say $\zeta$. 
Since both power series $\varphi_t(P^q_0(t))$ and $\varphi_t(P^q_1(t))$ have their coefficients in $\mathcal{O}_{\mathbb{L},S}$, we conclude that
$\zeta \in \mathcal{O}_{\mathbb{L},S}$. 
But $|\zeta|_v =1$ for all place $v$ over $\mathbb{L}$ by $(\star)$ hence it is a root of unity.

Note also that the equality $\varphi_t\left(P_0^q(t)\right)^{s_0} = \zeta\, \varphi_t\left(P_1^q(t)\right)^{s_1}$ holds as equality between adelic series, so that it
is also true for analytic functions at any place.

\smallskip

To conclude the proof of Theorem~\ref{tm:samegreen}, pick a place $v$ of $\LL$ and consider the connected component $U$ of 
$\{g_{0,v}>\tau_v/s_0\}=\{g_{1,v}>\tau_v/s_1\}$ in $C^{v,\an}$ whose closure in $\hat{C}$ contains $\mathfrak{c}$.  We need to argue that $P^q_{c,a}(c_0)$ and $P^q_{c,a}(c_1)$ belongs to the domain of
convergence of the B\"ottcher coordinate $\varphi_{v,c,a}$ for any $c,a\in U$.

Recall that $s_0g_{0,v}(c,a) = s_1g_{1,v}(c,a)$ for some positive integers $s_0, s_1$.
It follows that 
\begin{multline*}
\min\{g_{c,a}(P^q_{c,a}(c_0)), g_{c,a}(P^q_{c,a}(c_1))\} 
= 
3^q 
\min\{g_{c,a}(c_0), g_{c,a}(c_1)\} 
> \\
3^q \min\left\{\frac{s_0}{s_1}, \frac{s_1}{s_0}\right\}\,
\max\{g_{c,a}(c_0), g_{c,a}(c_1)\} 
>\\
G(c,a) + \max\{s_0g_{c,a}(c_0), s_1g_{c,a}(c_1)\} 
>
G(c,a) + \tau_v
\end{multline*}
for $q$ large enough
and we conclude by Proposition~\ref{prop:green is green}.

The proof of Theorem~\ref{tm:samegreen} is now complete.

\section{Special curves having a periodic orbit with a constant multiplier}\label{sec:perm}
In this section, we prove Theorem~\ref{thm:perm}.

Pick an integer $m\ge1$, a complex number $\lambda\in \C$, and consider the set of polynomials $P_{c,a}$ that admits a periodic orbits of period $m$ and multiplier $\lambda$. It follows from~\cite[p. 225]{Silverman} that this set is an algebraic curve in $\poly_3$ (see also~\cite[Appendix D]{milnor3}, \cite[Theorem 2.1]{BB2} or \cite[\S 6.2]{favregauthier}). Let us be more precise:

\begin{theorem}[Silverman]\label{siverman}
For any integer $m\geq1$, there exists a polynomial $p_m\in\Q[c,a,\lambda]$ with the following properties.
\begin{enumerate}
\item For any $\lambda\in\C\setminus\{1\}$,  $p_m(c,a,\lambda)=0$ if and only if $P_{c,a}$ has a cycle of exact period $m$ and multiplier $\lambda$.
\item When $\lambda =1$, then $p_m(c,a,1)=0$ if and only if there exists an integer $k$ dividing $m$ such that $P_{c,a}$ has a cycle of exact period $k$ whose multiplier is a primitive $m/k$-th root of unity.
\end{enumerate}
\end{theorem}

We now come to the proof of Theorem~\ref{thm:perm}.

\medskip

One implication is easy. For any integer $m\ge1$, the curve $\Per_m(0)$ is contained in the union of the two curves $\{(c,a)\in\C^2\, ; \ P_{c,a}^m(c_0)=c_0\}$ and $\{(c,a)\in\C^2\, ; \ P_{c,a}^m(c_1)=c_1\}$. According to lemma~\ref{lm:per(n,k)infintepcf} below, it contains infinitely many post-critically finite parameters.

\begin{lemma}\label{lm:per(n,k)infintepcf}
Pick $n\geq0$, $k>0$ and $i\in\{0,1\}$. Any irreducible component $C$ of the set $\{(c,a),\, P_{c,a}^{n+k}(c_i)=P_{c,a}^n(c_i)\}$ contains infinitely many post-critically finite parameters.
\end{lemma}

\begin{proof}
We argue over the complex numbers, and use the terminology and results from~\cite{favredujardin}. 
In particular, a critical point $c_i$, $i=0,1$ is said to be active at a parameter $(c,a)$ if 
the family of analytic functions $P^n_{c,a}(c_i)$ is normal in a neighborhood of $(c,a)$. \label{page-pfB}

\smallskip

Suppose that $C$ is an irreducible component of the set
 \[\{(c,a),\, P_{c,a}^{n+k}(c_i)=P_{c,a}^n(c_i)\}\]
 where $n\geq0$, $k>0$ and $i\in\{0,1\}$. To fix notation we suppose $i=0$.
Observe that $g_{c,a}(c_0) =0$ on $C$, and since $G(c,a) = \max\{g_{c,a}(c_0), g_{c,a}(c_1)\}$ is a proper function 
on $\poly_3$ (see Proposition~\ref{prop:growth Green}) it follows that $g_{c,a}(c_1)$ is also proper on $C$. 
In particular, $c_1$ has an unbounded orbit when $c,a\in C$ is close enough to infinity in $\poly_3$. 
It follows from e.g.~\cite[Theorem~2.5]{favredujardin} (which builds on~\cite[Theorem 2.2]{McMullen4})
that $c_1$ is active at at least one point $(c_0,a_0)$ on $C$. The arguments of~\cite[Lemma 2.3]{favredujardin}
based on Montel's theorem show that $(c_0,a_0)$ is accumulated by parameters for which $c_1$ is pre-periodic to a repelling cycle,
hence by post-critically finite polynomials. In particular, it contains infinitely many post-critically finite parameters.
\end{proof}

\medskip

For the converse implication, we proceed by contradiction and suppose that we can find a complex number $\lambda \neq 0$, an integer $m\ge 1$, and an irreducible component $C$ of $\per_m(\lambda)$ containing infinitely many post-critically finite polynomials.

Observe that, whenever $0<|\lambda|\leq 1$, any parameter $(c,a)\in C\subset \Per_m(\lambda)$ has a non-repelling cycle which is not super-attracting. In particular, at least one of its critical points has an infinite forward orbit (see e.g.\cite{Milnor4}). It follows that $\Per_m(\lambda)$ contains no post-critically finite parameter when  $0<|\lambda|\leq 1$. This argument is however not sufficient to conclude in general. But we shall see that a combination  of this argument applied at a place of residual characteristic $3$  together with the study of the explosion of multipliers on a branch at infinity of $C$  gives a contradiction.

\begin{proposition}\label{prop:estimate-mult}
Suppose $C$ is an irreducible component of $\Per_m(\lambda)$ with $\lambda \in \C^*$ and $m\ge1$ containing infinitely many post-critically finite polynomials. 
Then one of the two critical points is persistently periodic on $C$ and $\lambda$ is equal to the multiplier of a repelling periodic orbit of a post-critically finite quadratic polynomial.
\end{proposition}

We may thus assume that the curve $C$ is included in  $\Per_n(0)$ for some integer $n\ge1$ (possibly with $n\neq m$). Observe that the equation
$P^n_{c,a}(c_0) = c_0$ (resp. $P^n_{c,a}(c_1) = c_1$) is equivalent to the vanishing of a polynomial of the form
$\sqrt{3}^{1-3^n} a^{3^n} + \mathrm{l.o.t}$ (resp. $\sqrt{3}^{1-3^n} (a^3 - \frac{c^3}6)^{3^n} + \mathrm{l.o.t}$). 
It follows that the closure of $C$ in $\overline{\poly_3}$ intersects the line at infinity
in a set included in $\{[1:0:0], [\zeta:1:0]\}$ with $\zeta^3=6$ (see also~\cite[Theorem 4.2]{BB2}).

Consider the curve of unicritical polynomials  $c_0 = c_1$, which is defined by the equation  $c=0$. 
It intersects the line at infinity at $[0:1:0]$, so that Bezout' theorem implies the existence of 
a parameter $(c,a)\in C$ which is unicritical.

We conjugate $P_{c,a}$ by a suitable affine map to a polynomial $Q(z) = z^3 + t$. 
This unicritical polynomial has a periodic critical orbit. Proposition~\ref{prop:multiplier} below implies $|\lambda|_v <1$
at any place $v$ of residual characteristic $3$.
By the previous proposition, $Q$ also has a periodic orbit whose multiplier is equal to the multiplier
of a repelling orbit of a quadratic polynomial having a preperiodic critical point. 
Proposition~\ref{prop:multiplier} now gives $|\lambda|_v =1$ for this place, hence a contradiction.

The proof of Theorem~\ref{thm:perm} is complete.

\begin{proposition}\label{prop:multiplier}
Suppose $Q(z) = z^d + t$ is a post-critically finite unicritical polynomial of degree $d\ge 2$, and let $\lambda\neq 0$
be the multiplier of some periodic orbit of $P$.

Then $\lambda$ belongs to some number field $\KK$, and 
given any non-Archimedean place $v$ of $\KK$ we have:
\begin{itemize}
\item
$|\lambda|_v<1$ if the residual characteristic of $\KK_v$ divides $d$;
\item
$|\lambda|_v  =1$ if the residual characteristic of $\KK_v$ is prime to $d$.
\end{itemize}
\end{proposition}
\begin{proof}
Since $Q$ is post-critically finite, $t$ satisfies a polynomial equation with integral coefficients
hence belongs to a number field. Its periodic point are solutions of a polynomial of the form $Q^n(z)-z$  so that 
the periodic points of $Q$ and their multipliers also belong to a number field. We may thus fix a number field containing $t$, $\lambda$, and
 fix a place $v$ of $\KK$ of residual characteristic $p\ge 2$.
 
Observe that the completion of the algebraic closure of the completion of $\KK$ with respect to the norm induced by $v$
is a complete algebraically closed normed field isometric to the $p$-adic field $\C_p$.
We consider the action of $Q$ on the Berkovich analytification of the affine plane over that field.
To simplify notation we denote by $|\cdot|$ the norm on $\C_p$.

Suppose that  $|t|>1$. Then we have $ |Q(0)| = |t| > 1$, and thus $|Q^n(0)| = |Q^{n-1}(0)|^d = |t|^{d^n}\to \infty$ by an immediate induction.
This would imply the critical point to have an infinite orbit contradicting our assumption that $Q$ is post-critically finite.

We thus have $|t|\le1$. This implies that any point having a bounded orbit lies in the closed unit ball $\{z, |z| \le 1\}$.
Indeed the same induction as before yields  $|Q(z)| = |z|^d$ and  $|Q^n(z)| = |Q^{n-1}(z)|^d = |z|^{d^n}\to \infty$ for any $|z|>1$.

Pick any periodic point $w$ of period $k$ with multiplier $\lambda = (Q^k)'(w)\neq 0$. Observe that $Q'(z) = dz^{d-1}$.

Suppose first that $p$ divides $d$ so that $|d| <1$.
Since $|Q^j(w)|\le 1$ for all $j\ge0$ by what precedes, we have
\[|\lambda|=\prod_{j=0}^{k-1}|Q'(Q^j(w))| \le |d|^k <1~.\]
Suppose now that $p$ is prime to $d$, hence $|d|=1$. 
Observe that one has $Q(B(z)) = B(Q(z))$ for any $|z|\le 1$ where  $B(z) = \{w, \, |z-w|<1\}$ is the open ball of center $z$ and radius $1$. Since the critical point $0$ has a finite orbit,  two situations may arise. 

Either $B(0)$ is strictly preperiodic, and thus cannot contain any periodic orbit.
Or $B(0)$ is periodic, and is contained in the basin of attraction of some attracting periodic orbit. 
Since $0$ has a finite orbit, it has to be periodic. 
In both cases this implies the orbit of $w$ to be included in the annulus $\{|z|=1\}$.

We thus have
\[|\lambda| = \prod_{i=0}^{k-1} |Q'(Q^i(w))|=\prod_{i=0}^{k-1} |d\, (Q^i(w))^{d-1}| =  1~~,\]
which concludes the proof.
\end{proof}

\begin{proof}[Proof of Proposition~\ref{prop:estimate-mult}]
Since $C$ contains infinitely many post-critically finite polynomials we may assume it is defined over a number field $\KK$.
Let $\hat{C}$ be the normalization of the completion of $C$ in $\overline{\poly_3}$. 
Pick any branch $\mathfrak{c}$ of $C$ at infinity (i.e. a point in $\hat{C}$ which projects to the line at infinity in $\overline{\poly_3}$).
By Proposition~\ref{prop:defineL} we may choose an isomorphism of complete local rings $\widehat{\mathcal{O}_{\hat{C},\mathfrak{c}}}\simeq \mathbb{L}[[t]]$
such that $c(\mathsf{n}(t)), a(\mathsf{n}(t))$ are adelic series, i.e. formal Laurent series with coefficients in $\mathcal{O}_{\mathbb{L},S}((t))$ that are analytic at all places.

In the remaining of the proof, we fix an Archimedean place, and embed $\mathbb{L}$ into the field of complex numbers (endowed with its standard norm). 
We may suppose  $c(\mathsf{n}(t)), a(\mathsf{n}(t))$  are holomorphic in $0 < |t| < \epsilon$ for some $\epsilon$, and meromorphic at $0$. 
We get a one-parameter family of cubic polynomials $P_t:= P_{c(\mathsf{n}(t)), a(\mathsf{n}(t))}$ parameterized by the punctured disk $\D^*_\epsilon= \{0<|t|<\epsilon\}$.

Consider the subvariety $Z:= \{(z,t), \, P^m_t(z)=z\}\subset \C \times \D^*_\epsilon$. The projection map $Z \to \D_\epsilon^*$ is a finite cover which is unramified if $\epsilon$ is chosen small enough. By reducing $\epsilon$ if necessary, and replacing $t$ by $t^N$, we may thus assume that 
$Z\to \D^*_\epsilon$ is a trivial cover. In other words, there exists a meromorphic function $t\mapsto p(t)$ such that $P^m_t(p(t)) = p(t)$ and $(P^m_t)'(p(t)) = \lambda$.

\smallskip

As in Section~\ref{sec:curves in P3}, we denote by $\mathsf{P}(z)\in \C((t))[z]$ the cubic polynomial induced by the family $P_t$. 
It induces a continuous map on the analytification $\A^{1,\an}_{\C((t))}$, for which the point 
$\mathsf{p}\in \A^1(\C((t)))$ corresponding to $p(t)$ is periodic of period $m$ with multiplier
$(\mathsf{P}^m)'(\mathsf{p}) = \lambda$.
Observe that $\mathsf{P}$ has two critical points $\mathsf{c}_0$ and $\mathsf{c}_1$
 corresponding to the meromorphic functions $0$ and $c(\mathsf{n}(t))$ respectively.

\begin{lemma}
If $\mathsf{c}_0$ is not pre-periodic for $\mathsf{P}$, then 
$|\mathsf{P}^q(\mathsf{c}_0)|_t$ tends to infinity when $q\to \infty$.
\end{lemma}
\begin{proof}
Observe that our assumption is equivalent to the fact that $c_0$ is not persistently pre-periodic on $C$.

We claim that $g_0(t):= g_{P_t}(c_0)$ tends to infinity when $t\to0$.

Suppose first that $c_1$ is persistently pre-periodic on $C$. Then the function $g_1$ is identically zero on $C$, so that 
$G|_C = \max\{g_0, g_1\}|_C = g_0$. Since $G$ is  proper by Proposition~\ref{prop:growth Green}, and $(c(\mathsf{n}(t)), a(\mathsf{n}(t)))$ tends to infinity in $\overline{\poly_3}$ when $t\to 0$,
 we conclude that $g_0(t) \to \infty$.

When $c_1$ is not persistently pre-periodic on $C$, the two functions $g_0(t)$ and $g_1(t) := g_{P_t}(c_1)$ 
are proportional on $\mathfrak{c}$  by Theorem~\ref{tm:samegreen} (1). As before $\max\{g_0, g_1\} \to \infty$ as $t\to 0$
so that again $g_0(t) \to \infty$. 

\smallskip

By Proposition~\ref{prop:branch},  we can find $a >0$ such that 
$g_0(P_t) = a \log |t|^{-1} + O(1)$. And~\cite[Lemma 6.4]{favredujardin}
 implies\footnote{observe that the statement of the lemma is incorrectly stated in~\cite{favredujardin}, and the constant $C$ is  actually \emph{independent} on $P$.} 
 the existence of a constant $C>0$ such that  
 $g_{P_t}(z) \le \log \max \{ |z| , |c(\mathsf{n}(t))|, |a(\mathsf{n}(t))| \} + C$  for all $t$.
Since $g_{P_t} \circ P_t = 3 g_{P_t}$, we conclude that for all  $q\ge1$
\[\log \max \{ |P^q_t(c_0)| , |c(\mathsf{n}(t))|, |a(\mathsf{n}(t))| \} \ge 3^q g_t(0) - C=  3^q a \log |t|^{-1} + O(1)~.\]
This implies $|P^q_t(c_0)|_t \ge 3^q a |t|_t \to \infty$ when $q\to\infty$ as required.
\end{proof}

We continue the proof of Proposition~\ref{prop:estimate-mult}.
Suppose neither $c_0$ nor $c_1$ is persistently pre-periodic so that the previous lemma applies to both critical points. 
Translating its conclusion over the non-Archimedean field $\C((t))$, we get that $\mathsf{P}^q(\mathsf{c}_0)$ and 
$\mathsf{P}^q(\mathsf{c}_1)$ both tend to infinity when $q\to \infty$.
We may thus  apply~\cite[Theorem 1.1 (ii)]{Kiwi:cubic}, and \cite[Corollary 1.4]{Kiwi:cubic} (which is directly inspired
from a result of Bezivin). We conclude that all periodic cycles of $\mathsf{P}$ are repelling 
so that $|(\mathsf{P}^m)'(\mathsf{p})|_t >1$. This contradicts $|\lambda|_t =1$.

\smallskip

Suppose next that $c_0$ is persistently pre-periodic (which implies $c_1$ not to be persistently pre-periodic). 
Then $\mathsf{c}_0$ is pre-periodic whereas $\mathsf{c}_1$ escapes to infinity by the previous lemma.
Observe that if  $\mathsf{c}_0$ is eventually mapped to a point in the Julia set of $\mathsf{P}$, then \cite[Theorem 1.1 (iii) (a)]{Kiwi:cubic} combined with \cite[Corollary 1.4]{Kiwi:cubic}
implies that all cycles of $\mathsf{P}$ are repelling which gives a contradiction.

We can thus apply~\cite[Theorem 1.1 (iii) (b)]{Kiwi:cubic} to $\mathsf{P}$, and
the preperiodic critical point $\mathsf{c}_0 (=0)$ is contained in a closed ball $B= \{\mathsf{z}\in \C((t)),\, |\mathsf{z}|_t \le r\}$ for some positive $r>0$ that is periodic of exact period $n$. 
Since $B$ is fixed by the polynomial $\mathsf{P}^n (z)= \sum_{j\ge 2} b_j z^j$ with coefficients $b_j \in\C((t))$, 
the radius $r$ satisfies an equation of the form $|b_j| r^j = r$ for some $j$ hence
$r = |t|_t^l$ for some $l\in \Q$.
To simplify the discussion to follow we do a suitable base change $t \to t^N$, and we conjugate $\mathsf{P}$ by the automorphism $\mathsf{z} \mapsto t^{-l} \mathsf{z}$ so that $B$ becomes the closed unit ball. 
Observe that $0$ remains a critical point of $\mathsf{P}$ after this conjugacy.
 
Recall that the closed unit ball $B$ defines the Gauss point  $\mathsf{x}_g\in \A^{1,\an}_{\C((t))}$ for which we have
\[|\mathsf{Q} (\mathsf{x}_g)| := \sup_{z\in B} |\mathsf{Q}(z)|_t = \max |q_i| \text{ for all } \mathsf{Q} = \sum q_i z^i \in \C((t))[z]~.\]
Since $B$ is fixed by $\mathsf{P}^n$, it follows that $\mathsf{x}_g$ is also fixed by $\mathsf{P}^n$.
This is equivalent to say that $\mathsf{P}^n$ can be written as  $\mathsf{P}^n(z)= \sum_{i=1}^{3^n} a_i z^i$ where $\max |a_i| = 1$.

\smallskip

For any $\mathsf{z}\in \C((t))$ of norm $1$, denote by $\tilde{z}$ the unique complex number such that $|\mathsf{z} - \tilde{z}|_t <1$.

\begin{lemma}\label{lem:reduc}
We have $a_1 =0$, $|a_0| \le |a_2| = 1$, and $|a_i| <1$ for all $i\ge3$; and the complex quadratic polynomial
$\tilde{P}(z):= \widetilde{a_2} z^2 + \widetilde{a_0}$ has a preperiodic critical orbit.
\end{lemma}

\begin{lemma}\label{lem:orbit-p}
The orbit of the periodic point $\mathsf{p}$ intersects the ball $B$.
\end{lemma}

Replacing $\mathsf{p}$ by its image by a suitable iterate of $\mathsf{P}$ we may suppose that it belongs to $B$, i.e. $|\mathsf{p}|_t \le1$.
In fact we have $|\mathsf{P}^i(\mathsf{p})|_t =1$ for all $i\ge0$. 

Indeed if it were not the case, then the open unit ball would be periodic. Since it contains a critical point, it would be contained in the basin of attraction of an attracting periodic orbit which yields a contradiction. 

Observe also that the period of $\mathsf{p}$ is necessarily a multiple of $n$, say $nk$ with $k\ge1$. 

To render the computation of the multiplier of $\mathsf{p}$
easier, we conjugate $\mathsf{P}^n$ by  $z \mapsto a_2 z$. Since $|a_2|=1$, we still have $|\mathsf{p}|_t =1$, and the equality
$a_2 =1$ is now satisfied. 

By Lemma~\ref{lem:reduc}, we get $\sup_B |\mathsf{Q}| <1$ with $\mathsf{Q}:= \mathsf{P}^{n} - \widetilde{P}$, so that
\[(\mathsf{P}^{nk})'(\mathsf{p}) =  \prod_{i=0}^{k-1}(\mathsf{P}^{n})'(\mathsf{P}^{ni}(\mathsf{p}))=
(\widetilde{P}^{k})'(\widetilde{p})~.\]
But the multiplier of $\mathsf{p}$ is equal to $\lambda \in \C$. Hence it is equal to the multiplier of a repel\-ling periodic orbit of some
quadratic polynomial (namely $\tilde{P}$) having a preperiodic critical orbit, as was to be shown.
\end{proof}

\begin{proof}[Proof of Lemma~\ref{lem:reduc}]
The point $0$ is critical for $\mathsf{P}$ hence $a_1 =0$. 

Since the Gauss point is fixed by  $\mathsf{P}^n$, we have $\max_{i\ge 2} |a_i| =1$.
Let $d\ge2$ be the maximum over all integers $i$ such that $|a_i|= 1$.  The number of critical points
of $\mathsf{P}^n$ lying in the closed unit ball (counted with multiplicity) is precisely equal to $d-1$.

Since the exact period of $\mathsf{x}_g$ is $n$, and the other point escapes to infinity, the ball 
$B$ contains a unique critical point of $\mathsf{P}^n$ namely $0$. It follows that $d =2$, and $|a_2| =1 > \max_{i\ge3} |a_i|$.

Finally  $0$ is preperiodic by $\mathsf{P}^n$, hence the complex quadratic polynomial
$\widetilde{P^n}$ has a preperiodic critical orbit.
\end{proof}

\begin{proof}[Proof of Lemma~\ref{lem:orbit-p}]
Since the multiplier of $\mathsf{p}$ is $\lambda \in \C$, its $t$-adic norm is $1$, hence a small ball $U$ centered at $\mathsf{p}$ of positive radius is included in the filled-in Julia set of $\mathsf{P}$. By~\cite[Corollary 4.8]{Kiwi:cubic}, $U$ is eventually mapped into $B$, hence the claim.
\end{proof}

\section{A polynomial on a special curve admits a symmetry}\label{sec:symmetry}

We fix $\KK$ a number field, and $s_0, s_1$ two positive integers such that $s_0$ and $s_1$ are coprime.
We shall say that a cubic polynomial $P:= P_{c,a}$ with $c,a$ in a finite extension $\LL$ of $\KK$ satisfies the condition $(\mathfrak{P})$\label{def:dag} if 
the following holds: 
\begin{itemize}
\item[($\mathfrak{P}1$)]
For any place $v$ of $\LL$, we have $s_0g_{P,v}(c_0) = s_1g_{P,v}(c_1)$.
\item[($\mathfrak{P}2$)]
Given any place $v$ of $\LL$, if $\min\{g_{P,v}(P^n(c_0)), g_{P,v}(P^n(c_1))\} > G_v(P) + \tau_v$ for some integer $n\ge1$, then 
$$\frac{\varphi_{P,v}(P^n(c_0))^{s_0}}{\varphi_{P,v}(P^n(c_1))^{s_1}}$$ is a root of unity lying in $\KK$.
\end{itemize}

\smallskip

Recall the definition of the constant $\tau_v := \tau(\LL_v)$ from Proposition~\ref{prop:green is green}.

Observe that if the condition in $(\mathfrak{P}2)$ never occurs, then the normalized heights by $P$  
of both sequence of points $P^n(c_0)$ and $P^n(c_1)$ are bounded, hence $P$ is post-critically finite. 
We prove here the following

\begin{theorem}\label{thm:cubic-pol}
Suppose $P = P_{c,a}$ is a cubic polynomial defined over a number field $\LL$ satisfying the assumptions $(\mathfrak{P})$ which is not post-critically finite and such that 
$\min\{g_{P,v}(P^q(c_0)), g_{P,v}(P^q(c_1))\} > G_v(P) + \tau_v$ for some integer $q$ and some place $v$ of $\LL$.

Then there exists a root of unity $\zeta\in \KK$, an integer $q' \le C(\KK,q)$, and an integer $m\ge0$ such that the polynomial $Q(z) := \zeta P^m(z) + (1-\zeta) \frac{c}2$ commutes with 
all iterates $P^k$ such that $\zeta^{3^k} = \zeta$, and either $Q(P^{q'}(c_0)) = P^{q'}(c_1)$, or $Q(P^{q'}(c_1)) = P^{q'}(c_0)$.
\end{theorem}

\begin{remark}
We shall prove along the way that there exists an integer $k\geq1$ with $\zeta^{3^k}=\zeta$ so that the commutativity statement is non empty.
\end{remark}

\subsection{Algebraization of adelic branches at infinity}

The material of this section is taken from~\cite{Xie}. Let $\KK$ be a number field. For any place $v$ on $\KK$, denote by $\KK_v$ the completion of $\KK$ w.r.t. the $v$-adic norm. 

We cover the line at infinity $H_\infty$ of the compactification of the affine space $\A^2_\KK = \textup{Spec}\,\KK[x,y]$ by $\p^2_\KK$ by charts $U_\alpha = \spec \KK[x_\alpha, y_\alpha]$ centered at $\alpha\in H_\infty(\KK)$ such that $\alpha = \{(x_\alpha, y_\alpha) = (0, 0)\}$, $H_\infty\cap U_\alpha =\{x_\alpha =0\}$, and $x_\alpha =1/x$, $y_\alpha =y/x+c$ for some $c\in \KK$ (or $x_\alpha =1/y$, $y_\alpha = x/y$).

\smallskip

Fix $S$ a finite set of places of $\KK$. By definition, an adelic branch $\mathfrak{s}$  at infinity defined over the ring  $\mathcal{O}_{\KK,S}$
is a formal branch based at a point $\alpha\in H_\infty(\KK)$ given in coordinates $x_\alpha, y_\alpha$ as above by a formal Puiseux series 
\[y_\alpha = \sum_{j\geq1} a_j x_\alpha^{j/m} \in \mathcal{O}_{\KK,S}[[x_\alpha^{1/m}]]\]
such that $\sum_{j\geq1} a_j x^{j}$ is an adelic series.

Observe that for any place $v\not\in S$, then the radius of convergence is a least $1$. In the sequel, we set $r_{\mathfrak{s},\alpha,v}$ to be the minimum between $1$ and the radius of convergence over $\KK_v$ of this Puiseux series.
Any adelic branch $\mathfrak{s}$ based at $\alpha$ at infinity thus defines an analytic curve in an (unbounded) open subset of $\A^{2,\an}_v$:
\[Z^v (\mathfrak{s}) := \left\{(x_\alpha,y_\alpha) \in U_\alpha(\KK_v)\, ; \ y_\alpha^{m} = \sum_{j\geq1}a_{j} x_\alpha^{j}~, \ 0 < |x_\alpha|_v< r_{\mathfrak{s},\alpha,v}\right\}~.\]

\begin{theorem}[Xie]\label{tm:xie}
Suppose $\mathfrak{s}_1,\ldots, \mathfrak{s}_l$ are adelic branches at infinity, and let  $\{B_v\}_{v\in M_\KK}$  be a set of positive real numbers such that $B_v = 1$ for all but finitely many places.

Assume that there exists a sequence of distinct points $p_n = (x_n,y_n) \in \A^2(\KK)$ such that for all $n$ and for each place $v\in M_\KK$ then either we have $\max\{|x_n|_v,|y_n|_v\}\leq B_v$ or $p_n\in\cup_{i=1}^lZ^v(\mathfrak{s}_i)$. 

Then there exists an irreducible algebraic curve $Z$ defined over $\KK$ such that any branch of $Z$ at infinity is contained in the set $\{\mathfrak{s}_1,\ldots,\mathfrak{s}_l\}$ and $p_n$ belongs to $Z(\KK)$ for all $n$ large enough.
\end{theorem}

\subsection{Construction of an invariant correspondence}
Our aim is to prove the following statement. 
\begin{theorem}\label{thm:symmetry}
Suppose $P = P_{c,a}$ is a cubic polynomial satisfying the assumptions $(\mathfrak{P})$.
Then there exists a (possibly reducible) algebraic curve
$Z_P \subset \A^1 \times \A^1$ such that:
\begin{enumerate}
\item 
$\phi(Z_P) = Z_P$ with  $\phi(x,y):=(P(x), P(y))$;
\item
for all $n$ large enough, we have $(P^n(c_0), P^n(c_1))\in Z_P$;
\item
any branch at infinity of $Z_P$ is given by an equation
$\varphi_P(x)^{s_0} = \zeta\cdot\varphi_P(y)^{s_1}$ for some root of unity $\zeta \in \KK$.
\end{enumerate}
\end{theorem}

\begin{proof}
The proof is a direct application of Xie's theorem. 
Recall that the set $U_\KK$ of roots of unity that is contained in the number field $\KK$ is finite.

Recall that for each place $v$ over $\LL$, we let $g_{P,v} := \lim_n \frac1{3^n}\log^+|P^n|_v$ be the Green function of $P$, and write
$G_v(P) = \max\{g_{P,v}(c_0), g_{P,v}(c_1)\}$.

\begin{lemma}\label{lem:rightform}
For any $\zeta \in U$, there exists an adelic branch $\mathfrak{c}_\zeta$ based at a point $q \in H_\infty(\LL)$ 
such that for any place $v$ the analytic curve $Z^v (\mathfrak{c}_\zeta)$ is defined by the equation 
$\{\varphi_{P,v}(x)^{s_0} = \zeta\cdot \varphi_{P,v}(y)^{s_1}\}$ in the range
$\min\{|x|_v, |y|_v\} > \exp(G_v(P)+\tau_v)$.
\end{lemma}

Define $(x_n,y_n) := (P^{n}(c_0), P^{n}(c_1))\in \A^2(\LL)$, and consider the family of all adelic curves
$\mathfrak{c}_\zeta$ given by Lemma~\ref{lem:rightform} for all $\zeta \in U_\KK$.  We shall first check that 
all hypothesis of Xie's theorem are satisfied. 

To do so pick any integer $n$ and any place $v$ on $\LL$. Suppose first that $g_{P,v}(c_0) =0$. 
Since $g_{P,v}(P^{n}(c_0)) = 3^{n} g_{P,v}(c_0)=0$, we get 
$|x_n|_v\le e^{C_v}=: B_v$ by Lemma~\ref{lem:classic-estim}. The same upper bound applies to $|y_n|_v$  since
$g_{P,v}(c_1) =0$ by $(\mathfrak{P}1)$
so that 
$\max\{|x_n|_v,|y_n|_v\}\le B_v$ in this case.
Observe that $B_v=1$ for all but finitely many places $v$ of $\mathbb{L}$ by Lemma~\ref{lem:classic-estim}.

Suppose now that $g_{P,v}(c_0) >0$ so that $g_{P,v}(c_1) >0$ by $(\mathfrak{P}1)$. Fix $N$ large enough such that 
$g_{P,v}(P^N(c_0)) >G_v(P) + \tau_v$ and $g_{P,v}(P^N(c_1)) >G_v(P) + \tau_v$.
Then $P^N(c_0)$ and $P^N(c_1)$ lie in the domain of definition of the B\"ottcher coordinate  by Proposition~\ref{prop:green is green}. Since 
\[g_{P,v}(P^{n}(c_0)) = 3^{n-N} g_{P,v}(P^N(c_0))\ge g_{P,v}(P^N(c_0))> G_v(P)~,\]
we may also evaluate $\varphi_{P}$ at $x_n$  for all $n\ge N$. The same holds for $y_n$ and we get that
\[\frac{\varphi_{P}(x_n)^{s_0}}{\varphi_{P}(y_n)^{s_1}}
\]
is a root of unity $\zeta\in \KK$ by $(\mathfrak{P}2)$
hence $(x_n,y_n)$ belongs to $Z^v(\mathfrak{c}_{\zeta})$ for all $n\ge N$.

\smallskip

Xie's theorem thus applies to the sequence $\{(x_n,y_n)\}_{n\ge N}$, and we get an irreducible curve $Z \subset \A^1 \times\A^1$ that contains infinitely many points
$(x_n,y_n)$ and such that each of its branch at infinity is equal to $\mathfrak{c}_\zeta$ for some $\zeta\in U_\KK$.

\smallskip

Recall that $\phi(x,y)=(P(x),P(y))$, and pick any integer $n\ge1$. Then $\phi^n(Z)$ is an irreducible curve defined over $\LL$ whose branches at infinity
are the images under $\phi^n$ of the branches at infinity of $Z$.

Fix $\zeta\in U_\KK$ and pick $(x,y)\in Z^v(\mathfrak{c}_\zeta)$. Then $(x',y') = (P(x),P(y))$ satisfies
\[\frac{\varphi_{P}(x')^{s_0}}{\varphi_{P}(y')^{s_1}}
=
\frac{\varphi_{P}(x)^{3s_0}}{\varphi_{P}(y)^{3s_1}}
=\zeta^3
~,\]
hence $\phi(\mathfrak{c}_\zeta) = \mathfrak{c}_{\zeta^3}$. We conclude that any branch at infinity of $\phi^n(Z)$ is of the form $\mathfrak{c}_\zeta$ for some $\zeta\in U_\KK$.
Since two irreducible curves having a branch at infinity in common are equal, we see that $Z$ is pre-periodic  for the morphism $\phi$ so that $\phi^{l+k}(Z) = \phi^k(Z)$ for some $l, k>0$. Setting $Z_P:= \cup_{j=k}^{l+k-1}\phi^j(Z)$, we obtain a (possibly reducible) curve defined over $\LL$ such that 
$\phi(Z_P) = Z_P$ and $(x_n,y_n)\in Z_P$ for all $n\ge k$. This concludes the proof of the theorem.
\end{proof}

\begin{proof}[Proof of Lemma~\ref{lem:rightform}]
Recall from Lemma~\ref{lm:bottcher1} that \[\varphi_P(z)=\om\left(z-\frac{c}{2}\right)+\sum_{k\geq1}\frac{a_k}{z^{k}},\]
is an adelic series at infinity in the sense of \S\ref{sec:adelic series} , and therefore
\[\varphi^{-1}_P(z)=\frac1\om\, z + \frac{c}2+\sum_{k\geq1}\frac{b_k}{z^{k}},\]
too by Lemma~\ref{lem:adelic series}. We may assume that  $a_k, b_k \in \mathcal{O}_{\KK,S}$. Recall from Proposition~\ref{prop:green is green} that $\varphi_{P,v}$ induces an analytic isomorphism between
$\{z,\, g_{P,v}(z) > G_v(P)+\tau_v\}$ and $\{z', \, |z'|_v > \exp(G_v(P)+\tau_v)\}$.
By Lemma~\ref{lem:adelic series} the formal map $\varphi^{-1}_P$ defines an adelic
series at infinity in the terminology of \S\ref{sec:adelic series}. For each place $v$, this series coincides with the inverse map of $\varphi_P$
on the complement of the closed disk  of radius $\exp(G_v(P)+\tau_v)$ hence its domain of convergence is exactly
 $\{z', \, |z'|_v > \exp(G_v(P)+\tau_v)\}$.

It follows that 
\[Z_v:= \{(x,y),\, \varphi_P(x)^{s_0} =  \zeta\, \varphi_P(y)^{s_1}\}\]
defines an analytic curve in the domain $\min\{g_{P,v}(x),g_{P,v}(y)\}>G_v(P)+\tau_v$,
whose image under the isomorphism $(x',y') :=(\varphi_{P,v}(x),\varphi_{P,v}(y))$ is given by
\[Z'_v:= \{(x',y'),\, (x')^{s_0} = \zeta\, (y')^{s_1}\}\]
where $\min\{|x'|_v,|y'|_v\}>\exp(G_v(P)+\tau_v)$.

Pick any $\xi\in\bar{\Q}$ such that $\xi^{s_1} \zeta =1$. Let  $\mathfrak{c}_\zeta$
be the adelic branch at infinity defined by the formal Laurent series $(\varphi_P^{-1}(t^{-s_1}), \varphi_P^{-1}(\xi \, t^{-s_0}))$.
For any place $v$, the analytic curve $Z^v(\mathfrak{c}_\zeta)$ is included in $Z_v$. 
Since $s_0$ and $s_1$ are coprime, for any pair $(x',y')$ with $(x')^{s_0} = \zeta\, (y')^{s_1}$ 
and $\min\{|x'|_v,|y'|_v\}>\exp(G_v(P)+\tau_v)$, there exists $0<|t|_v<\exp\left(-\frac{G_v(P)+\tau_v}{\min\{s_0,s_1\}}\right)$ such that $x' = t^{-s_1}$ and $y' = \xi t^{-s_0}$. 

\smallskip

This proves that $Z^v(\mathfrak{c}_\zeta) = Z_v$ for all place as required.
\end{proof}

\subsection{Invariant correspondences are graphs}

Let $Z_0, \ldots, Z_{p-1}$ be the irreducible components of $Z_P$ such that $\phi(Z_i) = Z_{i+1}$ (the index computed modulo $p$). Since we assumed $P$ not to be post-critically finite, it is non-special  in the sense of \cite{pakovich}. 
We may thus apply Theorem~4.9 of op. cit. (or \cite[Theorem~6.24]{medvedev-scanlon}) to the component $Z_0$ of $Z_P$ that 
is $\phi^p$-invariant. It implies that after exchanging $x$ and $y$ if necessary, $Z_0$ is the graph of a polynomial map, i.e.
$Z_0 = \{ (Q(t),t)\}$ for some $Q\in \LL[t]$ such that $Q\circ P^p = P^p \circ Q$. Observe that by~\cite{Julia} the two polynomials $P$ and $Q$ share  a common iterate
 when $\deg(Q)\ge2$ since we assumed $P$ not to be post-critically finite.

\smallskip

We now work at an Archimedean place. Recall that the branch at infinity of $Z_0$ is of the form $\varphi_P(x)^{s_0} = \zeta \varphi_P(y)^{s_1}$ for some $\zeta \in U_\KK$.
Since $s_0$ and $s_1$ are coprime, it follows that $s_0 =1$ and $s_1 = \deg(Q)$, and therefore $s_1$ is a power of $3$, say $s_1 = 3^m$.
We get
\begin{equation}\label{eq:end of pf}
\varphi_P(Q(t)) = \zeta \varphi_P(t)^{3^m} = \zeta \varphi_P(P^m(t))~.
\end{equation}
for all $t$ of large enough norm. 
By Lemma~\ref{lm:bottcher1}, we get that 
$\varphi_P(t) = \om \left( t - \frac{c}2\right) + o(1)$ so that 
\begin{equation}
 \omega \left( Q(t) - \frac{c}2\right) = \om\zeta \left(P^m(t) - \frac{c}2\right) + o(1)
\end{equation}
which implies $Q(t) :=\zeta P^m(t) + (1-\zeta) \frac{c}2$ since 
a polynomial which tends to $0$ at infinity is identically zero.

\smallskip

At this point, recall our assumption that $\min\{g_{P,v}(P^q(c_0)), g_{P,v}(P^q(c_1))\} > G_v(P) + \tau_v$ for some integer $q$ and some place $v$ of $\LL$.
Then by $(\mathfrak{P}2)$
$\varphi_P(P^q(c_0))^{s_0} = \xi \varphi_P(P^q(c_1))^{s_1}$ for some root of unity $\xi\in \KK$
which implies  $\varphi_P(P^{q+n}(c_0)) = \xi^{3^n} \varphi_P(P^{q+n}(c_1))^{3^m}$. Since for some $n$ large enough
the point $(P^{q+n}(c_0), P^{q+n}(c_1))$ belongs to $Z_0$, we get $\xi^{3^n} = \zeta$.
Now observe that the least integer $n$ such that $\xi^{3^{n}} = \zeta$ is less that the cardinality of $U_\KK$. 
We get the existence of $q' \le C(\KK,q)$ such that $\varphi_P(P^{q'}(c_0)) = \zeta \varphi_P(P^{q'}(c_1))^{3^m}$.
Since $\varphi_P$ is injective, the equation~\eqref{eq:end of pf} shows that   $P^{q'}(c_0) = Q(P^{q'}(c_1))$.

\smallskip

Observe that $\zeta^{3^p}=\zeta$. Indeed, since $Z_0$ is $\phi^p$-invariant and since $\phi(\mathfrak{c}_\zeta)=\mathfrak{c}_{\zeta^3}$, we get $\mathfrak{c}_{\zeta^{3^p}}=\mathfrak{c}_{\zeta}$, hence $\zeta^{3^p}=\zeta$.

We now pick \emph{any} integer $k\geq1$ such that $\zeta^{3^k} = \zeta$. Then for all $t$ large enough, we have 
\[\varphi_P (Q \circ P^k(t)) = \zeta \varphi_P(P^k(t))^{3^m} = \zeta \varphi_P^{3^{m+k}}(t)\]
whereas
\[\varphi_P (P^k \circ Q(t))= \varphi_P (Q(t))^{3^k} = \zeta^{3^k} \varphi_P^{3^{k+m}}(t)~.\]
Since $\varphi_P$ is injective on a neighborhood of $\infty$, and since $\zeta^{3^k}=\zeta$ by assumption, we conclude that $Q\circ P^k = P^k \circ Q$.

\smallskip

This concludes the proof of Theorem~\ref{thm:symmetry}.

\section{Classification of special curves}\label{sec:thmA}

In this section, we prove Theorems~\ref{tm:alaBDM} and~\ref{tm:classification}.

\medskip

Before starting the proofs,  let us introduce some notation.
Pick $q,m\ge0$ and $\zeta$ a root of unity. We let
$Z(q,m,\zeta)$ be the algebraic set of those $(c,a) \in \A^2$ such that the polynomial $Q_{c,a}:=\zeta P_{c,a}^m + (1-\zeta) \frac{c}2$ commutes with all iterates $P_{c,a}^k$ of $P_{c,a}$ such that $\zeta^{3^k} = \zeta$, and either $Q_{c,a}(P_{c,a}^q(c_0)) = P_{c,a}^q(c_1)$, or $Q_{c,a}(P_{c,a}^q(c_1)) = P_{c,a}^q(c_0)$.

Observe that, when $m\ge1$, $Q$ has degree $3^m$ and when $(c,a)$ belongs to a fixed normed field $K$ then the Green function $g_{Q} := \lim_n \frac1{3^{nm}} \log^+|Q^n|$ is equal to $g_{c,a}$. Indeed since
$Q$ and $P^k$ commute they have the same filled-in Julia set, $\K_Q$ coincides with the filled-in Julia set $\K_P$ of $P$. 
And $g_{P}$ (resp. $g_{Q}$) is the unique continuous sub-harmonic function $g$ on $\A^{1,\an}_K$ that is zero on $\K_P$, harmonic
outside, with a logarithmic growth at infinity $g_{P} (z) = \log |z| +O(1)$ (resp. $g_{Q} (z) = \log |z| +O(1)$). As $\K_P=\K_Q$, this gives $g_{c,a}=g_{Q}$.

\subsection{Proof of Theorem~\ref{tm:alaBDM}}

The implication $(1)\Rightarrow (3)$ is exactly point $1.$ of Theorem~\ref{tm:samegreen}. 

\smallskip

The implication $(3)\Rightarrow (4)$ follows from Corollary~\ref{cor:distrib} when $s_0$ and $s_1$ are both nonzero. Indeed, by assumption the critical point $c_0$ is pre-periodic for $P_{c,a}$ if and only if $c_1$ is. Moreover we observe that $c_0$ is active at at least one parameter of $C$, since the function $G$ is proper, hence that the set for which it is pre-periodic is infinite, by e.g. \cite[Lemma 2.3]{favredujardin}.

When $s_1=0$, 
then $g_{0,v} \equiv 0$ on $C$ at all places.
By \cite[Theorem 2.5]{favredujardin} there exist $n>0$ and $k\geq0$ such that $C$ is an irreducible component of $\{(c,a)\in\A^2\, ; \ P_{c,a}^{n+k}(c_0)=P_{c,a}^k(c_0)\}$. 
By Theorem~\ref{tm:bifurcationheights} (applied to arbitrary weights) the family of functions $\{g_{1,v}\}_{v\in M_\KK}$ induces a semi-positive adelic metric
on some ample line bundle on the normalization of the completion of $C$ so that 
Thuillier-Yuan's theorem applies. This gives $(4)$ by observing that $g_0+g_1=g_1$. 
The case $s_0=0$ is treated similarly.

\smallskip

The implication $(4) \Rightarrow (1)$ is obvious.

\smallskip

To prove $(2) \Rightarrow (1)$, we observe that if $c_0$ is not persistently pre-periodic on $C$ then it is active at at least one parameter by \cite[Theorem 2.5]{favredujardin}
and that the set of parameters for which it is pre-periodic is infinite by e.g. \cite[Lemma 2.3]{favredujardin}.

\smallskip

We now prove $(3) \Rightarrow (2)$.  We suppose $c_0$ is not persistently pre-periodic on $C$. 
Pick some parameter $(c,a)\in C$ and suppose $c_0$ is pre-periodic.
We need to show that $P_{c,a}$ is post-critically finite.  

Since $c_0$ is not persistently pre-periodic on $C$ we have $s_1\neq 0$ (again by \cite[Theorem 2.5]{favredujardin} applied at any Archimedean place). 
In the case $s_0 =0$ then $c_1$ is persistently pre-periodic and $P_{c,a}$ is clearly post-critically finite. 
We may thus assume that $s_0$ and $s_1$ are both non-zero and the functions $g_{0,v}, g_{1,v}$ vanish on the same set in $C^{v,\an}$ for any place $v$ of $\KK$.
Observe that $c_0$ being pre-periodic implies $(c,a)$ to be defined
over a number field. It follows that for all the Galois-conjugates $(c',a')$ of $(c,a)$ (over the defining field $\KK$ of the curve $C$)
we have $G_v(c',a') := \max\{g_{0,v}(c',a'),g_{1,v}(c',a')\} =0$.
It follows from~\cite{Ingram} or~\cite[Theorem 3.2]{favregauthier} that $P_{c,a}$ is post-critically finite.

\smallskip

Let us now prove $(3)\Rightarrow (5)$. We suppose $C$ is special.
If either $c_0$ or $c_1$ is persistently pre-periodic in $C$, the assertion $(5)$ holds true with $\zeta=1$ and $i=j$ by \cite[Theorem 2.5]{favredujardin}.

Assume from now on that we are not in this case. Replacing $\KK$ by a finite extension
we may assume that all roots of unity $\zeta$ appearing in Theorem~\ref{thm:cubic-pol} 2. belong to $\KK$, 
since there are at most the number of branches at infinity of $C$ of such roots of unity.

Let $\mathcal{B}$ be the set of all $(c,a)\in C(\LL)$ where $\LL$ is  a finite extension of $\KK$ such that $P_{c,a}$ is not post-critically finite. Given a place $v$ of $\KK$
we also define the subset 
$\mathcal{B}_v$ of $\mathcal{B}$ of parameters $c,a$ such that 
$g_{0,v}(c,a)>0$. 
This set is infinite  since post-critically finite polynomials form a bounded set in $C^{v,\an}$.
\begin{lemma}\label{lem:P2}
For any $(c,a)\in \mathcal{B}$, the polynomial $P_{c,a}$ satisfies $(\mathfrak{P}1)$ and $(\mathfrak{P}2)$.
\end{lemma}
Pick $q$ large enough such that $3^q > \max\{ s_0/s_1, s_1/s_0\}$, and 
fix a place $v$ of residual characteristic $\ge 5$. 
Now choose any $(c,a)\in\mathcal{B}_v$.  Then $g_{1,v}(c,a)$ is also positive
and $\min\{ g_{c,a,v}(P^q(c_0)),  g_{c,a,v}(P^q(c_1))\} > G_v(c,a)$ so that Theorem~\ref{thm:cubic-pol} applies by the previous lemma.
We get a positive integer $q'$ (bounded from above by a constant $C$ depending only on $\KK$ and $q$), a root of unity $\zeta\in \KK$  
and an integer $m\ge0$ such that  $(c,a)\in Z({q'},m,\zeta)$. 

Since $g_Q =g_P$, and $Q(P^{q'}(c_0)) = P^{q'}(c_1)$ we have 
\[3^mg_{P,v} (P^{q'}(c_0))= g_{P,v} (Q(P^{q'}(c_0))) = g_{P,v}(P^{q'}(c_1))\]
so that $3^m = \frac{s_0}{s_1}$. We conclude that the algebraic set consisting of the union of the curves $Z({q'},m,\zeta)$ with  $3^m = \frac{s_0}{s_1}$, $q'\le C$
and $\zeta$ ranging over all roots of unity lying in $\KK$ contains $\mathcal{B}_v$. 

It follows that $C$ is an irreducible component of one of these curves.

\smallskip

To end the proof of the theorem, we are left with proving $(5) \Rightarrow (3)$. Suppose that $C$ is an irreducible component of $Z(m,q,\zeta)$ for some $m\ge0$ and $q\geq0$ and some root of unity $\zeta$. Observe that $Z(q,m,\zeta)$ hence $C$ are defined over a number field say $\KK$.

When $m\ge1$, then for all place $v$ of that number field we have $g_{Q_{c,a},v}= g_{P_{c,a},v}$ for all $(c,a)\in C(\LL)$ for some finite extension $\LL$ of $\KK$. In particular
$Q_{c,a}(P_{c,a}^q(c_i)) = P_{c,a}^q(c_j)$ implies $3^{m} g_{i,v} (c,a)= g_{j,v}(c,a)$ which proves $(3)$ (with $s_0=0$ or $s_1=0$
if $i=j$).  When $m=0$ and $\zeta \neq 1$ and $C \neq \{ c=0\}$, then $Q_{c,a}(c_0) = (1-\zeta)c/2 \neq 0$ 
hence $i\neq j$. It follows that $ g_{P_{c,a},v} \circ Q_{c,a}=  g_{P_{c,a},v}$ hence $g_{0,v} = g_{1,v}$.
When $C = \{ c=0\}$, then $c_0 = c_1$ so that again $g_{0,v} = g_{1,v}$.
Finally when $m=0$ and $\zeta =1$, then $i\neq j$ and $P_{c,a}^q(c_0) = P_{c,a}^q(c_1)$ hence $g_{0,v} = g_{1,v}$ at all places.

\begin{proof}[Proof of Lemma~\ref{lem:P2}]
Pick $(c,a)\in \mathcal{B}$.  By the point 1. of Theorem~\ref{tm:samegreen}, $P_{c,a}$ satisfies $(\mathfrak{P}1)$. 
To check $(\mathfrak{P}2)$, we need to introduce a few sets. Fix any place $v$ of $\KK$, and for any integer $n\ge 0$, 
define the open subset of $C^{v,\an}$
$$
\Omega_{n,v} := \{(c',a'),\, \min\{g_{c',a',v}(P^n_{c',a'}(c_0)),  g_{c',a',v}(P^n_{c',a'}(c_1))\} > G_v(c',a') + \tau_v\}~.
$$
On  $\Omega_{n,v}$ one can define the analytic map
$$
M_n(c',a'):= \frac{\varphi_{c',a'}(P^n_{c',a'}(c_0))^{s_0}}{\varphi_{c',a'}(P^n_{c',a'}(c_1))^{s_1}}~.
$$
Observe that $\Omega_{n+1,v} \subset \Omega_{n,v}$, and  $M_{k+l}(c',a') = M_{k}(c',a')^{3^l}$ on $\Omega_{k,v}$ for all integers $k,l\ge 0$.
We also define the increasing sequence of open sets
$$
U_{n,v} := \left\{(c',a'),\, G_v(c',a') > \frac{\tau_v}{3^n-1}\right\} \subset C^{v,\an}~.
$$
Since $G_v$ is subharmonic and proper on $C^{v,\an}$, the set $U_{n,v}$ contains no bounded component by the maximum principle.
\begin{lemma}\label{lem of lem}
Suppose $3^r \ge \max\{s_0/s_1, s_1/s_0\}$. Then
we have $\Omega_{n,v} \subset U_{n,v}$ and $U_{n,v} \subset \Omega_{n+r,v}$.
\end{lemma}
By the point 2. of Theorem~\ref{tm:samegreen}, one can find two integers $q\ge1$ and $N\ge1$ such that $M_q$ is well-defined and constant equal to a root of unity lying in $\KK$
in each component of $U_{N,v}$.

Let $V$ be the connected component of $\Omega_{n,v}$ containing $(c,a)$. This open set might or might not be bounded. 
By the previous lemma, if $n\ge \max\{r+N,q\}$, then $U_{n-r,v}\subset\Omega_{n,v}$ so that   
$M_{n}$ is well-defined on $U_{n-r,v}$. Since all components of $U_{n-r,v}$ are unbounded, and
$M_n = M_q^{3^{n-q}}$ in $U_{N,r}$, we conclude that $M_n$ is locally constant in $U_{n-r,v}$ (hence on $V$)
equal to a root of unity lying in $\KK$. 

When $n\le n_0 = \max\{r+N,q\}$, then $(M_n)^{3^{n_0-n}} = M_{n_0}$ which we know is constant in $V$ equal to a root of unity
lying in $\KK$. We conclude that $M_n$ is constant on $V$ equal to a root of unity lying in a fixed extension $\KK'$ of $\KK$ that only depends
on the constants $r, N$ and $q$. Since these constants are in turn independent of the place $v$, we conclude the proof of the lemma
replacing $\KK$ by $\KK'$.
\end{proof}

\begin{proof}[Proof of Lemma~\ref{lem of lem}]
Pick $(c',a')\in \Omega_{n,v}$. We may suppose that $G_v(c',a') = g_{c',a',v}(c_0)$ so that 
$G_v(c',a')= g_{c',a',v}(c_0) = \frac1{3^n}g_{c',a',v}(P^n_{c',a'}(c_0)) > \frac1{3^n} (G_v(c',a') + \tau_v)$
which implies $(c',a')\in U_{n,v}$.

Conversely suppose  $(c',a')\in U_{n,v}$. As before we may suppose that $G_v(c',a') = g_{c',a',v}(c_0)$ so that
\begin{multline*}
g_{c',a',v}(P^{n+r}(c_0)) \ge g_{c',a',v}(P^n(c_0)) = 3^n  g_{c',a',v}(c_0) =\\ 3^n G_v(c',a') = G_v(c',a') + (3^n-1) G_v(c',a') > G_v(c',a') + \tau_v~.
\end{multline*}
Similarly we have
$$g_{c',a',v}(P^{n+r}(c_1)) = 3^{n+r}\frac{s_0}{s_1}\, g_{c',a',v}(c_1)\ge 3^n g_{c',a',v}(c_0)> G_v(c',a') + \tau_v$$
hence $(c',a')\in \Omega_{n+r,v}$.
\end{proof}

\subsection{Proof of Theorem~\ref{tm:classification}}

According to the implication $(1) \Rightarrow (5)$ of Theorem~\ref{tm:alaBDM},  any irreducible algebraic curve $C$ of $\poly_3$ containing infinitely many 
post-critically finite polynomials is a component of some $Z(q,m,\zeta)$ so that Theorem~\ref{tm:classification} reduces to the following.

\begin{proposition}\label{curves Z(q,m,zeta)}
~
\begin{enumerate}
\item The set $Z(q,m,1)$ is equal to the union $\{P_{c,a}^{m+q}(c_1) = P_{c,a}^q(c_0)\} \cup \{P_{c,a}^{m+q}(c_0) = P_{c,a}^q(c_1)\}$, hence contains an algebraic curve. Moreover, one has $Z(1,0,1)=Z(0,0,1)$. 
\item The set $Z(q,m, -1)$ is infinite if and only if $m=0$, and  we have $Z(q,0,-1)=\{12a^3-c^3-6c=0\}$ for any $q\geq0$.
\item if $\zeta^2\neq1$, the set $Z(q,m,\zeta)$ is finite.
\end{enumerate}
\end{proposition}

We shall rely on the following observation. Denote by $\crit(P)$ the set of critical points of the polynomial $P$.

\begin{lemma}\label{lm:permute-critic}
Pick any $(c,a)\in Z(q,m,\zeta)$, and suppose that $$Q_{c,a}= \zeta P_{c,a}^m + (1-\zeta)\frac{c}2$$ is a polynomial that 
commutes with $P_{c,a}^k$ and 
$\zeta$ is a $(3^k-1)$-root of unity. 
Then we have
\[Q_{c,a}(\crit(P_{c,a}^{k+m}))=Q_{c,a}(\crit(P_{c,a}^m))\cup \crit(P_{c,a}^k)~.\]
\end{lemma}

\begin{proof}
Write $P = P_{c,a}$ and $Q = Q_{c,a}$.
Differentiate the equality $P^k\circ Q=Q\circ P^k$. Since $Q'=\zeta\cdot (P^m)'$, we get 
\begin{multline*}
\crit(Q\circ P^k) = P^{-k}(\crit(Q))\cup \crit(P^k) = \\
 P^{-k}(\crit(P^m))\cup \crit(P^k)=
\crit(P^{k+m})~,
\end{multline*}
and therefore
\[\crit(P^{k+m})=\crit(P^k\circ Q) = \crit(P^m)\cup Q^{-1}(\crit(P^k))~,\]
and we conclude taking the image of both sides by $Q$.
\end{proof}

We now come to the proof of the Proposition.

\begin{proof}[Proof of Proposition~\ref{curves Z(q,m,zeta)}]
We may and shall assume that all objects are defined over the field of complex numbers. 

\medskip

\noindent {\bf 1}. Suppose $Z(q,0,\zeta)$ contains an irreducible curve $C$. We shall prove that either $\zeta = \pm 1$, or $C=\{c=0\}$.   

Observe that for any $(c,a)\in Z(q,0,\zeta)$, the polynomial $Q_{c,a}$ is an affine map which commutes with $P_{c,a}^k$, hence $g_{c,a}(Q_{c,a}(z))=g_{c,a}(z)$ for all $z\in\C$. Without loss of generality, we may suppose that $Q_{c,a}(P^q_{c,a}(c_0))=P^q_{c,a}(c_1)$, hence $G(c,a)=g_{0}(c,a)=g_1(c,a)$.

Suppose that $Z(q,0,\zeta)$ contains an irreducible curve $C$. If $g_0$ vanishes identically on $C$ then $g_1$ also, and this implies both critical points to be persistently preperiodic so that all polynomials in $C$ are post-critically finite. This cannot happen, so that  we can find an open set $U$ in $C$ such that $G(c,a)>0$ for all $(c,a)\in U$. 

Pick any parameter $(c,a)$ in $U$. We have $Q_{c,a}(\crit(P_{c,a}^k))=\crit(P_{c,a}^k)$ by Lemma~\ref{lm:permute-critic}, so that $Q_{c,a}(c_0),Q_{c,a}(c_1)\in \crit(P_{c,a}^k)$.
Since
\[\crit(P_{c,a}^k)=\bigcup_{0\leq j\leq k-1}P_{c,a}^{-j}(\crit(P_{c,a}))\]
we get $g_{c,a}(\alpha)=3^{-j}g_{c,a}(c_0)<g_{c,a}(c_0)=G(c,a)$  for any $\alpha$ lying  in $\crit(P_{c,a}^k)$ but not in $\crit(P_{c,a})$. 
However  $g_{c,a}(Q_{c,a}(c_0))=g_{c,a}(Q_{c,a}(c_0))=G(c,a)$, therefore we have $Q_{c,a}(c_0),Q_{c,a}(c_1)\in \crit(P_{c,a})=\{c_0,c_1\}$. 

This implies either $(1-\zeta)c=0$, or $(1+\zeta)c=0$, hence $\zeta =\pm 1$ or $C=\{c=0\}$ as required.

\medskip

\noindent {\bf 2}. Suppose now that $C$ is an irreducible curve included in $Z(q,m,\zeta)$ with $m>0$. We claim that either $\zeta = 1$, or $C=\{c=0\}$ as above.   

We proceed similarly as in the previous case. We suppose that $Z(q,m,\zeta)$ is infinite. For any $(c,a)\in Z(q,m,\zeta)$,  the polynomial $Q_{c,a}$ commutes with $P_{c,a}^k$ for some $k$, and has degree $3^m>1$.  In particular we have equality of Green functions $g_{Q_{c,a}}=g_{c,a}$. Without loss of generality we may (and shall) assume $Q_{c,a}(P^q_{c,a}(c_0))=P^q_{c,a}(c_1)$, which implies $g_0(c,a)=3^mg_1(c,a)$.

Assume now by contradiction that $\zeta\neq1$. Proceeding as in the previous case, we can find an open set $U\subset C$ such that $G(c,a)=g_0(c,a)>0$ for all $(c,a)\in U$. Pick now $(c,a)\in U$.
\begin{lemma}\label{lem:clarify pf}
For any $\alpha\in P_{c,a}^{-m}\{c_0\}$, we have $Q_{c,a}(\alpha)\in\{ c_0, Q_{c,a}(c_1)\}$.
\end{lemma}

Observe that
\[Q_{c,a}(\alpha)=\zeta P_{c,a}^m(\alpha)+(1-\zeta)\frac{c}{2}=\zeta c_0+(1-\zeta)\frac{c}{2}=(1-\zeta)\frac{c}{2}.\]
The equality $Q_{c,a}(c_1)=Q_{c,a}(\alpha)$ therefore gives 
$Q_{c,a}(c_1)=\zeta P_{c,a}^m(c_1)+(1-\zeta)\frac{c}{2}=(1-\zeta)\frac{c}{2}$, and we find $P_{c,a}^m(c_1)=0=c_0$
so that $C$ is a component of $Z(1,m,1)$.

The equality $c_0=Q_{c,a}(\alpha)$, implies $(1-\zeta)\frac{c}{2}=0$ so that either $\zeta =1$, or $C$ equals $\{c=0\}$.

\medskip

\noindent {\bf 3}. We have $Z(q,0,-1)=Z(0,0,-1)$ for all $q\geq0$.

Fix $q\ge0$, and pick any $(c,a)\in Z(q,0,-1)$. 
Observe that $(-1)^{3} = -1$ hence $Q_{c,a}(z) = - z +c$ commutes with $P_{c,a}$ by definition.
A direct computation shows that this happens if and only if $(c,a)$ belongs to the curve
$D_1:= \{12a^3-c^3-6c=0\}$. 

One can also check that $Q_{c,a}(c_0) = c_1$ for any parameter on $D_1$, and this implies
$(Q_{c,a} \circ P_{c,a}^q) (c_0) = P_{c,a}^q ( Q_{c,a}(c_0)) = P_{c,a}^q(c_1)$ for any $q\ge0$. 
This implies the claim.

\medskip

\noindent {\bf 4}. The irreducible curve $D_0= \{c=0\}$ is included in $Z(q,m,\zeta)$ if and only if $m=0$ and $\zeta =1$.

Observe that any polynomial $P:=P_{0,a}$ in $D_0$ is unicritical with a single critical point at $0$, 
so that $D_0$ is included in $Z(q,0,1)$ for all $q\ge0$. 
Observe also that  $g_0  = g_1 >0$ on a non-empty open subset of $D_0$.

Suppose that $D_0$ is included in $Z(q,m,\zeta)$ for some positive integer $m>0$.
Then the Green function of $Q:= Q_{0,a} $ is equal to $g_P$ and
the equation $Q(P^q(c_0)) = P^q(c_1)$ implies $g_P(c_0) = 3^m g_P(c_1) > g_P(c_1)$
at least when $P$ is close enough to infinity. This is absurd. 

Suppose now that  $D_0$ is included in $Z(q,0,\zeta)$ with $\zeta\neq 1$ so that 
$Q (z)= \zeta z$. One checks by induction that for any integer $k\ge1$ one has
\[P^k(z) = q_k z^{3^k} +  s_ka^3 z^{3^k-3} + \text{l.o.t}\]
with $q_k, s_k\in \Q_+^*$. Choose $k$ minimal such that $\zeta^{3^k} = \zeta$.  We get 
\[Q^{-1} \circ P^k \circ Q (z) =  q_k z^{3^k} +  s_ka^3 \zeta^{3^k -4} z^{3^k-3} + \text{l.o.t}\neq P^k\]
which yields a contradiction, and concludes the proof of our claim.

\medskip

\noindent {\bf 5}. We may now prove the proposition. The first statement follows from the definition of $Z(q,m,1)$, since  in that case we have $Q =P^m$ which always commutes with $P$. Moreover, the curve $Z(1,0,1)$ is given by the equation
\[0=P_{c,a}(c_0)-P_{c,a}(c_1)=a^3-\left(a^3-\frac{c^3}{6}\right)=\frac{c^3}{6}~,\]
whence $Z(1,0,1)=\{c=0\}=Z(0,0,1)$.

For the second statement, suppose first that $Z(q,m,-1)$ is infinite. 
By the second step, we have $m=0$, or $D_0 = \{c=0\}$ is included in $Z(q,m,-1)$. The fourth step rules out the latter possibility so that 
$m=0$. Conversely if $m=0$ we may apply the third step to conclude that $Z(q,0,-1)$ is a curve equal to $D_1 = \{12a^3-c^3-6c=0\}$. 

For the third statement, pick $\zeta \neq \pm 1$ and suppose by contradiction that $Z(q,m,\zeta)$ is infinite. The first and second step imply that 
$Z(q,m,\zeta)$ contains $D_0$ which is impossible by Step 4. 

This concludes the proof of the proposition.
\end{proof}

\begin{proof}[Proof of Lemma~\ref{lem:clarify pf}]
Take $\alpha\in P_{c,a}^{-m}\{c_0\}$, and observe that $\alpha\in \crit(P_{c,a}^{k+m})$. 
According to Lemma~\ref{lm:permute-critic}, we have
\[Q_{c,a}(\alpha)\in \crit(P_{c,a}^k)\cup Q_{c,a}(\crit(P_{c,a}^{m}))\]
and $g_{c,a}(Q_{c,a}(\alpha))=3^mg_{c,a}(\alpha)=3^m\cdot3^{-m}g_{c,a}(c_0)=g_{c,a}(c_0)=G(c,a)>0$. 

\smallskip

Pick any point $z\in \crit(P_{c,a}^k)\cup Q_{c,a}(\crit(P_{c,a}^{m}))$, and suppose it is equal to neither $c_0$ nor $Q_{c,a}(c_1)$. Then we are in one of the following (excluding) cases:
\begin{enumerate}
\item $z$ is a preimage of $c_0$ under $P_{c,a}^{j}$ for some $1\leq j\leq k-1$,  and $g_{c,a}(z)<g_P(c_0)$;
\item $z$ is a preimage of $c_1$ under $P_{c,a}^{j}$ for some $0\leq j\leq k-1$,  in which case $g_{c,a}(z)\leq g_{c,a}(c_1)<g_{c,a}(c_0)$;
\item $z\in Q_{c,a}(\crit(P_{c,a}^{m}))$, so that  $g_{c,a}(z)=3^mg_{c,a}(w)$ for some point $w\in \crit(P_{c,a}^m)=\bigcup_{0\leq j\leq m-1}P_{c,a}^{-j}(\crit(P_{c,a}))$. 
\end{enumerate}
In the last case two sub-cases arise. When $w$ is a preimage of $c_0$, we find 
\[ g_{c,a}(z) = 3^m g_{c,a}(w) \ge 3^m \frac1{3^{m-1}} g_{c,a}(c_0)> g_{c,a}(c_0)~.\]
Otherwise $w$ is a preimage of $c_1$ distinct from $c_1$ since $z\neq Q_{c,a}(c_1)$.
And we find $g_{c,a}(z) = 3^m g_{c,a}(w)\leq 3^{m-1} g_{c,a}(c_1)<g_{c,a}(c_0)$.

 Since $g_{c,a}(Q_{c,a}(\alpha)) = g_{c,a}(c_0)$ we conclude that $z\neq Q_{c,a}(\alpha)$ as required.
\end{proof}

\bibliographystyle{short}
\bibliography{biblio}

\end{document}